\documentclass[a4paper,reqno,12pt]{amsart}

\usepackage[all, 2cell, knot,]{xy} \UseAllTwocells \SilentMatrices
\usepackage{xspace}
\usepackage{amsmath}
\usepackage{amstext}
\usepackage{amsfonts}
\usepackage[mathscr]{euscript}
\usepackage{amscd}
\usepackage{latexsym}
\usepackage{amssymb}
\usepackage{layout}
\usepackage{amsthm}
\usepackage{amsfonts}
\usepackage{amsxtra}
\usepackage{color}
\usepackage[usenames,dvipsnames,svgnames,table]{xcolor}
\usepackage{graphics}
\usepackage{bm}
\usepackage{tabulary}
\usepackage{fancyhdr}
\usepackage[bookmarks=true,hyperfootnotes=false]{hyperref}
\hypersetup{
			colorlinks=true,
			linkcolor=blue,
			anchorcolor=black,
			citecolor=green,
			urlcolor=black,
}
\theoremstyle{plain}
\newtheorem{theorem}{Theorem}[section]
\newtheorem{lemma}[theorem]{Lemma}
\newtheorem{proposition}[theorem]{Proposition}
\newtheorem{corollary}[theorem]{Corollary}
\theoremstyle{definition}
\newtheorem{definition}[theorem]{Definition}
\newtheorem{example}[theorem]{Example}

\newtheorem{notation}[theorem]{Notation}
\newtheorem{remark}[theorem]{Remark}

\newcommand{\clA}{\mathcal{A}}

\newcommand{\clC}{\mathcal{C}}
\newcommand{\clD}{\mathcal{D}}
\newcommand{\clE}{\mathcal{E}}

\newcommand{\clI}{\mathcal{I}}

\newcommand{\clU}{\mathcal{U}}

\newcommand{\zg}{\gamma}

\newcommand{\zD}{\Delta}

\newcommand{\zve}{\varepsilon}

\newcommand{\zs}{\sigma}

\newcommand{\pt}{\partial}

\newcommand{\Ho}{\operatorname{Ho}}
\newcommand{\Dec}{\operatorname{Dec}}
\newcommand{\Id}{\operatorname{Id}}

\newcommand{\pr}{\operatorname{pr}}

\newcommand{\bsim}{/\!\!\sim}

\newcommand{\nid}{\noindent}
\newcommand{\ds}{\displaystyle}
\newcommand{\bk}{\bigskip}
\newcommand{\mk}{\medskip}

\newcommand{\ovl}[1]{\overline{#1}}
\newcommand{\ovll}[1]{\overset{=}{#1}}
\newcommand{\tms}[2]{\overset{#1}{\times}_{#2}}
\newcommand{\up}[1]{^{(#1)}}
\newcommand{\lo}[1]{_{(#1)}}
\newcommand{\rw}{\rightarrow}
\newcommand{\Rw}{\Rightarrow}
\newcommand{\lw}{\leftarrow}

\newcommand{\xrw}{\xrightarrow} 
\newcommand{\hxrw}[1]{\xymatrix{\ \ar@{^{(}->}^{#1}[r] & \ }}
\newcommand{\tiund}[1]{{\times}_{#1}\:}
\newcommand{\pro}[3]{#1\tiund{#2}\overset{#3}{\cdots}\tiund{#2}#1}
\newcommand{\tens}[2]{#1\,\tiund{#2}\,#1}
\newcommand{\uset}[2]{\underset{#1}{#2}}
\newcommand{\oset}[2]{\overset{#1}{#2}}

\newcommand{\mi}{\text{-}}
\newcommand{\nm}{(n-1)}
\newcommand{\bl}{\bullet}

\newcommand{\seqc}[3]{{#1}_{#2},...,{#1}_{#3}}

\newcommand{\dop}[1]{\Delta^{{#1}^{op}}}
\newcommand{\Dop}{\Delta^{op}}
\newcommand{\Dnop}{\Delta^{{n}^{op}}}
\newcommand{\Dmenop}{\Delta^{{n-1}^{op}}}
\newcommand{\cat}[1]{\mbox{$\mathsf{Cat^{#1}}$}}
\newcommand{\Cat}{\mbox{$\mathsf{Cat}\,$}}

\newcommand{\Gpd}{\mbox{$\mathsf{Gpd}$}}

\newcommand{\cathd}[1]{\mbox{$\mathsf{Cat_{hd}^{#1}}$}}

\newcommand{\catwg}[1]{\mbox{$\mathsf{Cat_{wg}^{#1}}$}}
\newcommand{\gcatwg}[1]{\mbox{$\mathsf{GCat_{wg}^{#1}}$}}
\newcommand{\tawg}[1]{\mbox{$\mathsf{Ta_{wg}^{#1}}$}}
\newcommand{\lta}[1]{\mbox{$\mathsf{LTa_{wg}^{#1}}$}}
\newcommand{\ftawg}[1]{\mbox{$\mathsf{FTa_{wg}^{#1}}$}}
\newcommand{\gtawg}[1]{\mbox{$\mathsf{GTa_{wg}^{#1}}$}}

\newcommand{\seg}[1]{\mbox{$\mathsf{Seg_{#1}}$}}
\newcommand{\segpsc}[2]{\mbox{$\mathsf{SegPs}$}\funcat{#1}{#2}}
\newcommand{\Ps}{\mbox{$\mathsf{Ps}$}}
\newcommand{\psc}[2]{\mbox{$\mathsf{Ps}$}\funcat{#1}{#2}}
\newcommand{\PsTalg}{\mbox{\sf Ps-}T\mbox{\sf -alg}}
\newcommand{\Top}{\mbox{$\mathsf{Top}$}}
\newcommand{\muk}{\mu_k}
\newcommand{\hmu}[1]{\hat\mu_{#1}}
\newcommand{\hmuk}{\hat{\mu}_k}

\newcommand{\Nn}{N_{(n)}}
\newcommand{\N}[1]{N_{(#1)}}
\newcommand{\Nb}[1]{N_{(#1)}}
\newcommand{\Nu}[1]{N^{(#1)}}
\newcommand{\di}[1]{d^{(#1)}}

\newcommand{\tld}{\tilde{d}}
\newcommand{\Dn}{D_{n}}
\newcommand{\Dnm}{D_{n-1}}
\newcommand{\D}[1]{D_{#1}}

\newcommand{\p}[1]{p^{(#1)}}

\newcommand{\op}[1]{\bar{p}^{(#1)}}
\newcommand{\q}[1]{q^{(#1)}}

\newcommand{\rz}{R_0}

\newcommand{\zgu}[1]{\zg^{(#1)}}
\newcommand{\Tan}{\mbox{$\mathsf{Ta^{n}}$}}
\newcommand{\ta}[1]{\mbox{$\mathsf{Ta^{#1}}$}}
\newcommand{\gta}[1]{\mbox{$\mathsf{GTa^{#1}}$}}
\newcommand{\Set}{\mbox{$\mathsf{Set}$}}
\newcommand{\St}{St\,}

\newcommand{\uk}{\underline{k}}
\newcommand{\ur}{\underline{r}}
\newcommand{\us}{\underline{s}}
\newcommand{\Qn}{Q_{n}}
\newcommand{\Qnm}{Q_{n-1}}
\newcommand{\Discn}{Disc_{n}}

\newcommand{\funcat}[2]{[\Delta^{{#1}^{op}},#2]}
\newcommand{\ps}{\sf{ps}}

\newcommand{\nfol}{$n$-fold }

\begin{document}

\title [\tiny{Weakly globular $n$-fold $n$-categories as...}]{Weakly globular $\bm{N}$-fold categories as a model of weak $\bm{N}$-categories}

\author[\tiny{Simona Paoli}]{Simona Paoli}
 \address{Department of Mathematics, University of Leicester,
LE17RH, UK}
 \email{sp424@le.ac.uk}

\date{10 September 2016}




\begin{abstract}
We study a new type of higher categorical structure, called weakly globular $n$-fold category, previously introduced by the author. We show that this structure is a model of weak $n$-categories by proving that it is  suitably equivalent to the Tamsamani-Simpson model. We also introduce groupoidal weakly globular $n$-fold categories and show that they are algebraic models of $n$-types.
\end{abstract}

\maketitle


\section{Introduction and Summary}\label{sec-intro}

Higher category theory is motivated and find applications to diverse areas, such as homotopy theory \cite{Be2} \cite{L2}, algebraic geometry \cite{Simp}, mathematical physics \cite{Lu3}, logic and computer science \cite{HTT},\cite{Voe}.
Higher categories generalize categories: the latter comprises objects and arrows, while higher categories admit higher arrows (also called higher cells) and compositions between them.

In a strict higher category the compositions of higher cells are associative and unital, like in a category. Although easy to define, strict higher categories are insufficient for many applications, and the broader class of weak higher categories is needed. In the latter, compositions are associative and unital only up to an invertible cell in the next dimension, and these associativity and unit isomorphisms are suitably compatible or coherent.

Making this intuition precise is quite complex: in low dimensions, it gave rise to the notions of bicategory \cite{Ben} and tricategory \cite{GPS}, in which the associtivity and unit isomorphisms and the relative coherence axioms are given explicitely. In dimension higher than $3$, the complexity of the structure necessitates a different approach: one in which a combinatorial machinery can encode the idea of a weak $n$-category while the coherence axioms are automatically satisfied.

Several different models of weak $n$-category exist, using a variety of techniques such as operads \cite{B}, \cite{Cheng2}, \cite{Lu2}, simplicial sets \cite{Simp}, \cite{Ta}, opetopes  \cite{Cheng1} and many others.

Several models have also been developed for  weak higher categories admitting cells in all dimensions, giving rise to  notions of $(\infty,n)$-category \cite{BeRe}, \cite{Be3}, \cite{Be2}, \cite{Bk1}, \cite{Lu3}, \cite{L2} as well as weak $\omega$-category \cite{Ve}.

In this work we concentrate on the 'truncated' case, with cells in dimensions $0$ up to $n$. This is close to one of the original motivations for the development of higher categories, namely the modelling of the building blocks of spaces, the $n$-types, and it is of fundamental importance for higher category theory. It also leads to applications to homotopy theory in the search for a combinatorial description of the $k$-invariants of spaces and of simplicial categories.

In \cite{Pa2},\cite{Pa3} the author introduced a new higher categorical structure, the category $\catwg{n}$ of weakly globular $n$-fold categories. In this paper we show that $\catwg{n}$ constitutes a model of weak $n$-categories. We show this by proving that weakly globular $n$-fold categories are suitably equivalent to one of the existing models of higher categories, the Tamsamani $n$-categories \cite{Ta}, \cite{Simp} and that they give a model of $n$-types in the higher groupoidal case. The latter (also called the homotopy hypothesis) is one of the main desiderata from a model of weak $n$-categories, while the comparison with the Tamsamani model is a contribution to the still largely open problem of comparing between different models of higher categories.

The category $\catwg{n}$ is based on the simple structure of iterated internal categories, or $n$-fold categories. This offers advantages in terms of applications. In forthcoming projects, we will exploit the $n$-fold nature of our structure to bridge between the simplicial and operadic approaches to higher categories, and we will develop algebraically defined cohomology theories for groupoidal weakly globular $n$-fold categories to study the $k$-invariants of spaces and of simplicial categories.

\subsection{The three Segal-type models} Our model lies in the context of three different Segal-type models: the weakly globular $n$-fold categories $\catwg{n}$ introduced by the author in \cite{Pa2}, the Tamsamani $n$-categories $\ta{n}$ developed by Tamsamani \cite{Ta} and Simpson \cite{Simp}, and the weakly globular Tamsamani $n$-categories $\tawg{n}$ introduced by the author in \cite{Pa3}. There are full and faithful embeddings:
\begin{equation*}
\xymatrix@R15pt @C30pt{
& \tawg{n} & \\
\Tan \ar@{^{(}->}[ur] & & \catwg{n} \ar@{_{(}->}[ul]\\
& n\mi\Cat \ar@{_{(}->}[ul]\ar@{^{(}->}[ur] &
}
\end{equation*}
The category $n\mi\Cat$ of strict $n$-categories admits a multi-simplicial description as the full subcategory of $\nm$-fold simplicial objects $X\in \funcat{n-1}{\Cat}$ satisfying the following
\begin{itemize}
  \item [(i)] $X_0 \in \funcat{n-2}{\Cat}$ and $X_{\uset{r}{1...1}0} \in \funcat{n-r-2}{\Cat}$ are discrete - that is constant multi-simplicial sets - for all $1\leq r \leq n-2$.

      \noindent Here we use  Notation \ref{not-simp}.
      \medskip

  \item [(ii)] The Segal maps (see Definition \ref{def-seg-map}) in all directions are isomorphisms.
\end{itemize}
The sets $X_0$ (resp. $X_{\uset{r}{1...1}0}$) in (i) correspond to the sets of 0-cells (resp. $r$-cells) for $1\leq r\leq n-2$; the sets of $\nm$ and $n$-cells are given by $X_{\uset{n-1}{1...1}0}$ and $X_{\uset{n}{1...1}}$ respectively.

The isomorphisms of the Segal maps (condition (ii)) ensures that the composition of cells is associative and unital.

The discreteness condition (i) is also called the globularity condition. The name comes from the fact that it determines the globular shape of the cells in a strict $n$-category. For instance, when $n=2$, we can picture 2-cells as globes
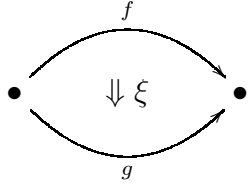
\begin{equation*}
\xymatrix{
\bullet \ar@/^2pc/^f[rr] \ar@/_2pc/_g[rr] &  \Downarrow \xi & \bullet
}
\end{equation*}
Strict $n$-categories have several applications, for instance in homotopy theory in the groupoidal case where they are equivalent to crossed $n$-complexes (see \cite{BHS}). However, they do not satisfy the homotopy hypothesis (see \cite{S2}) for a counterexample that strict 3-groupoids do not model 3-types).

Therefore we must relax the structure to obtain a model of weak $n$-category. Using the multi-simplicial framework, we consider three approaches to this:
\begin{itemize}
  \item [a)] In the first approach, we preserve the globularity condition (i) and we relax the Segal map condition (ii) by allowing the Segal maps to be suitably defined higher categorical equivalences. This makes the composition of cells no longer strictly associative and unital. This approach leads to the category $\ta{n}$ of weakly globular $n$-fold categories.
      \medskip

  \item [b)] In the second approach condition (ii) is preserved while the globularity condition (i) is replaced by weak globularity: the objects $X_0$, $X_{\uset{r}{1...1}0}$ ($1\leq r\leq n-2$) are no longer discrete but 'homotopically discrete' in a higher categorical sense that allows iterations. The notion of homotopically discrete $n$-fold category is a higher order version of equivalence relations. In particular, if $X$ is homotopically discrete, it is suitably equivalent to a discrete $n$-fold category $X^d$ via a map $\zg : X\rw X^d$. This approach leads to the category $\catwg{n}$ of weakly globular $n$-fold categories.
      \medskip

  \item [c)] In the third approach, both conditions (i) and (ii) are relaxed. This leads to the category $\tawg{n}$ of weakly globular Tamsamani $n$-categories.
\end{itemize}
The following are some common features of the three models, which we denote collectively by $\seg{n}$.
\begin{itemize}
  \item [(1)] $\seg{n}$ is defined inductively on dimension, starting with $\seg{1}=\Cat$ and $\seg{n}\subset \funcat{}{\seg{n-1}}$. In particular, unravelling this definition gives an embedding
      \begin{equation*}
        J_{n}:\seg{}\rw\funcat{n-1}{\Cat}\;.
      \end{equation*}
      \smallskip

  \item [(2)]The second common feature is the (weak) globularity condition. Namely, if $X\in\seg{n}$, then $X_0$ is a homotopically discrete $\nm$-fold category, and it is discrete if $X\in\Tan$. .
      \medskip

  \item [(3)] Given $X\in\seg{n}$, we can apply the functor isomorphism classes of objects functor $p:\Cat\rw\Set$ levelwise to $J_n X\in \funcat{n-1}{\Cat}$ to obtain $\ovl{p}J_n X\in\funcat{n-1}{\Set}$. We require that this is the multinerve of an object of $\seg{n-1}$; that is, there is a functor
      \begin{equation*}
        \p{n}:\seg{n}\rw\seg{n-1}
      \end{equation*}
      making the following diagram commute:
      \begin{equation*}
      \xymatrix{
      \seg{n} \ar@{^{(}->}^{J_n}[rr]  \ar_{\p{n}}[d] && \funcat{n-1}{\Cat} \ar^{\ovl{p}}[d] \\
      \seg{n-1} \ar@{^{(}->}_{\Nb{n-1}}[rr]  && \funcat{n-2}{\Set}
      }
      \end{equation*}
      The functor $\p{n}$, called $n^{th}$ truncation functor, is used to inductively define $n$-equivalences in $\seg{n}$.

      Given $X\in\seg{n}$ and $(a,b)\in X_0^d$, let $X(a,b)\subset X_1$ be the fiber of the map
      \begin{equation*}
        X_1 \xrw{(d_0,d_1)} X_0\times X_0 \xrw{\zg\times\zg} X_0^d\times X_0^d\;.
      \end{equation*}

      Then 1-equivalences in $\seg{n}$ are just equivalences of categories. Inductively, if we defined $\nm$-equivalences in $\seg{n-1}$, we say that a map $f:X\rw Y$ in $\seg{n}$ is a $n$-equivalence if for all $a,b \in X_0$,
      \begin{equation*}
        f(a,b):X(a,b)\rw Y(fa,fb)
      \end{equation*}
      are $\nm$-equivalences and $\p{n}f$ is a $\nm$-equivalence. This definition is a higher dimensional generalization of a functor which is fully faithful and essentially surjective on objects.
\medskip

  \item [(4)] Given $X\in \seg{n}$, since $X\in\funcat{}{\seg{n-1}}$ and there is a map $\zg:X_0\rw X_0^d$, we can consider the induced Segal maps for $k\geq 2$:
       \begin{equation*}
         \hmu{k}:X_k \rw \pro{X_1}{X_0^d}{k}\;.
       \end{equation*}
       In defining $\seg{n}$ we require these maps to be $\nm$-equivalences. Note that when $X\in\Tan$, $\zg=\Id$, so $\hmu{k}$ are just the Segal maps.
\end{itemize}

\subsection{Main results}\label{subs-intro-main}
Our main result, Corollary \ref{cor-the-disc-func}, is that there are comparison functors
\begin{equation*}
    \Qn:\Tan \leftrightarrows \catwg{n}:\Discn
\end{equation*}
inducing equivalences of categories
\begin{equation*}
    \Tan\bsim^n\;\simeq\;\Cat\bsim^n
\end{equation*}
after localization with respect to the $n$-equivalences. The functor $\Qn$, called rigidification, was introduced by the author in \cite{Pa3} while the functor $\Discn$, called discretization, is developed in this paper.

We also introduce (see Definition \ref{def-gta-1}) the full subcategory
\begin{equation*}
    \gcatwg{n}\subset \catwg{n}
\end{equation*}
of groupoidal weakly globular $n$-fold categories and we show (Corollary \ref{cor-gta-2}) that there is an equivalence of categories
\begin{equation*}
    \gcatwg{n}\bsim^n\;\simeq\;\Ho\mbox{\rm\text{(n-types)}}\;.
\end{equation*}
This means that our model of weak $n$-categories satisfies the homotopy hypothesis. A direct construction of the functor
\begin{equation*}
    G_n :\Top\rw\gcatwg{n}
\end{equation*}
was given by Blanc and the author in \cite{BP}.

\subsection{The discretization functor}\label{subs-intro-discr}
The idea of the functor $\Discn$ is to replace the homotopically discrete structures in $X\in\catwg{n}$ by their discretizations in order to recover the globularity condition. This affects the Segal maps, which from being isomorphisms in $X$ become $(n-1)$-equivalences in $\Discn X$.

We illustrate this in the case $n=2$. Given $X\in\catwg{2}$, by definition $X_0\in\cathd{}$, so there is a discretization map $\zg:X_0\rw X_0^d$ which is an equivalence of categories. Given a choice $\zg'$ of pseudo-inverse, we have $\zg\zg'=\Id$ since $X_0^d$ is discrete.

We can therefore construct $D_0 X\in\funcat{}{\Cat}$ as follows
\begin{equation*}
    (D_0 X)_k = \left\{
                  \begin{array}{ll}
                    X_0^d, & k=0 \\
                    X_k, & k>0\;.
                  \end{array}
                \right.
\end{equation*}
The face maps $(D_0 X)_1\rightrightarrows (D_0 X)_0$ are given by $\zg \pt_i$\; $i=0,1$ (where $\pt_i$ are the face maps for $X$) while the degeneracy map $(D_0 X)_0 \rw (D_0 X)_1$ is $\zs_0\zg'$. All other face and degeneracy maps are as in $X$. Since $\zg\zg'=\Id$, all simplicial identities are satisfied. By construction, $(D_0 X)_0$ is discrete while the Segal maps are given by
\begin{equation*}
    \pro{X_1}{X_0}{k}\rw \pro{X_1}{X_0^d}{k}
\end{equation*}
and these are equivalences of categories since $X\in\catwg{2}$. Thus, by definition, $D_0 X\in\ta{2}$. This construction however does not afford a functor $D_0:\catwg{2}\rw\ta{2}$, but only a functor
\begin{equation*}
    D_0:\catwg{2}\rw(\ta{2})_{\ps}
\end{equation*}
where $(\ta{2})_{\ps}$ is the full subcategory of $\psc{}{\Cat}$ whose objects are in $\ta{2}$. This is because, for any morphism $F:X\rw Y$ in $\ta{2}$, the diagram in $\Cat$
\begin{equation*}
\xymatrix@C=35pt{
X_0^d \ar^{f^d}[r] \ar_{\zg'(X_0)}[d] & Y_0^d \ar^{\zg'(Y_0)}[d]\\
X_0 \ar_{f}[r] & Y_0
}
\end{equation*}
in general only pseudo-commutes. Hence $D_0$ cannot be used as a definition of $Disc_2$. To overcome this problem we replace $\catwg{n}$ with a category $\ftawg{n}$ where there are functorial choices of sections to the discretization maps of the homotopically discrete structures, and we then build the discretization $D_n : \ftawg{n}\rw\ta{n}$ with the above construction (and its higher dimensional analogue).

The discretization functor $\Discn : \catwg{n}\rw\ta{n}$ is then the composite
\begin{equation*}
    \catwg{n}\xrw{G_n}\ftawg{n}\xrw{D_n}\ta{n}\;.
\end{equation*}
\subsection{Organization of the paper}\label{subs-intro-org}
In Section \ref{sec-prelim} we recall some basic background on (multi) simplicial techniques. In Section \ref{sec-segal-type-model} we recall the three Segal type models from \cite{Ta} and the author's works \cite{Pa1}, \cite{Pa2}, \cite{Pa3}, as well as the construction of the rigidification functor $\Qn$ from \cite{Pa3}.

In Section \ref{sec-canonical} we prove a technical result (Corollary \ref{cor-gen-const-1}) which is the basis of the construction of the category $\ftawg{n}$ of Section \ref{sec-fta}.

In Section \ref{sec-fta} (Proposition \ref{pro-fta-1}) we build the functor
\begin{equation*}
    G_n:\catwg{n}\rw \ftawg{n}
\end{equation*}
and we use this in Section \ref{sec-fta-to-tam} to construct the discretization functor
\begin{equation*}
    \Discn : \catwg{n}\rw \ta{n}\;.
\end{equation*}
In Section \ref{sec-group-wg-nfol-cat} we define the category $\gcatwg{n}$ of groupoidal weakly globular \nfol categories and we show that it is an algebraic model of $n$-types.


\bk

\textbf{Acknowledgements}: This work is supported by a Marie Curie International Reintegration Grant No 256341. I thank the Centre for Australian Category Theory for their hospitality and financial support during August-December 2015, as well as the University of Leicester for its support during my study leave.  I also thank the University of Chicago for their hospitality and financial support during April 2016.


\section{Preliminaries}\label{sec-prelim}
In this section we review some basic simplicial techniques that we will use throughout the paper as well as some categorical background on pseudo-functors and their strictification, and on a technique to produce pseudo-functors. The material in this section is well-known, see for instance \cite{Borc}, \cite{Jard}, \cite{lk}, \cite{PW}, \cite{Lack}.
\subsection{Simplicial objects}\label{sbs-simp-tech}
Let $\zD$ be the simplicial category and let $\Dnop$ denote the product of $n$ copies of $\Dop$. Given a category $\clC$, $\funcat{n}{\clC}$ is called the category of $n$-simplicial objects in $\clC$ (simplicial objects in $\clC$ when $n=1$).
\begin{notation}\label{not-simp}
    If $X\in \funcat{n}{\clC}$ and $\uk=([k_1],\ldots,[k_n])\in \Dnop$, we shall denote $X ([k_1],\ldots,[k_n])$ by $X(k_1,\ldots,k_n)$, as well as $X_{k_1,\ldots,k_n}$ and $X_{\uk}$. We shall also denote $\uk(1,i)=([k_1],\ldots,[k_{i-1}],1,[k_{i+1}],\ldots,[k_n]) \in \Dnop$ for $1\leq i\leq n$.

    Every $n$-simplicial object in $\clC$ can be regarded as a simplicial object in $\funcat{n-1}{\clC}$ in $n$ possible ways. For each $1\leq i\leq n$ there is an isomorphism
    \begin{equation*}
        \xi_i:\funcat{n}{\clC}\rw\funcat{}{\funcat{n-1}{\clC}}
    \end{equation*}
    given by
    \begin{equation*}
        (\xi_i X)_{r}(k_1,\ldots,k_{n-1})=X(k_1,\ldots,k_{i-1},r,k_{i+1},\ldots,k_{n-1})
    \end{equation*}
    for $X\in\funcat{n}\clC$ and $r\in\Dop$.
\end{notation}
\begin{definition}\label{def-fun-smacat}
    Let $F:\clC \rw \clD$ be a functor, $\clI$ a small category. Denote
    \begin{equation*}
        \ovl{F}:[\clI,\clC]\rw [\clI,\clD]
    \end{equation*}
    the functor given by
    \begin{equation*}
        (\ovl{F} X)_i = F(X(i))
    \end{equation*}
    for all $i\in\clI$.
\end{definition}
\begin{definition}\label{def-seg-map}
    Let ${X\in\funcat{}{\clC}}$ be a simplicial object in any category $\clC$ with pullbacks. For each ${1\leq j\leq k}$ and $k\geq 2$, let ${\nu_j:X_k\rw X_1}$ be induced by the map  $[1]\rw[k]$ in $\Delta$ sending $0$ to ${j-1}$ and $1$ to $j$. Then the following diagram commutes:
\begin{equation}\label{eq-seg-map}
\xymatrix@C=20pt{
&&&& X\sb{k} \ar[llld]_{\nu\sb{1}} \ar[ld]_{\nu\sb{2}} \ar[rrd]^{\nu\sb{k}} &&&& \\
& X\sb{1} \ar[ld]_{d\sb{1}} \ar[rd]^{d\sb{0}} &&
X\sb{1} \ar[ld]_{d\sb{1}} \ar[rd]^{d\sb{0}} && \dotsc &
X\sb{1} \ar[ld]_{d\sb{1}} \ar[rd]^{d\sb{0}} & \\
X\sb{0} && X\sb{0} && X\sb{0} &\dotsc X\sb{0} && X\sb{0}
}
\end{equation}

If  ${\pro{X_1}{X_0}{k}}$ denotes the limit of the lower part of the
diagram \eqref{eq-seg-map}, the \emph{$k$-th Segal map for $X$} is the unique map
$$
\muk:X\sb{k}~\rw~\pro{X\sb{1}}{X\sb{0}}{k}
$$
\noindent such that ${\pr_j\,\muk=\nu\sb{j}}$ where
${\pr\sb{j}}$ is the $j^{th}$ projection.
\end{definition}
\begin{definition}\label{def-ind-seg-map}

    Let ${X\in\funcat{}{\clC}}$ and suppose that there is a map $\zg: X_0 \rw X^d_0$ in $\clC$  $\zg: X_0 \rw X^d_0$  such that the limit of the diagram
\begin{equation*}
\xymatrix@R25pt@C16pt{
& X\sb{1} \ar[ld]_{\zg d\sb{1}} \ar[rd]^{\zg d\sb{0}} &&
X\sb{1} \ar[ld]_{\zg d\sb{1}} \ar[rd]^{\zg d\sb{0}} &\cdots& k &\cdots&
X\sb{1} \ar[ld]_{\zg d\sb{1}} \ar[rd]^{\zg d\sb{0}} & \\
X^d\sb{0} && X^d\sb{0} && X^d\sb{0}\cdots &&\cdots X^d\sb{0} && X^d\sb{0}
    }
\end{equation*}
exists; denote the latter by $\pro{X_1}{X_0^d}{k}$. Then the following diagram commutes, where $\nu_j$ is as in Definition \ref{def-seg-map}, and $k\geq 2$
\begin{equation*}
\xymatrix@C=20pt{
&&&& X\sb{k} \ar[llld]_{\nu\sb{1}} \ar[ld]_{\nu\sb{2}} \ar[rrd]^{\nu\sb{k}} &&&& \\
& X\sb{1} \ar[ld]_{\zg d\sb{1}} \ar[rd]^{\zg d\sb{0}} &&
X\sb{1} \ar[ld]_{\zg d\sb{1}} \ar[rd]^{\zg d\sb{0}} && \dotsc &
X\sb{1} \ar[ld]_{\zg d\sb{1}} \ar[rd]^{\zg d\sb{0}} & \\
X^d\sb{0} && X^d\sb{0} && X^d\sb{0} &\dotsc X^d\sb{0} && X^d\sb{0}
}
\end{equation*}
The \emph{$k$-th induced Segal map for $X$} is the unique map
\begin{equation*}
\hmuk:X\sb{k}~\rw~\pro{X\sb{1}}{X^d\sb{0}}{k}
\end{equation*}
such that ${\pr_j\,\hmuk=\nu\sb{j}}$ where ${\pr\sb{j}}$ is the $j^{th}$ projection.
\end{definition}
\subsection{$\mathbf{n}$-Fold internal categories}\label{sbs-nint-cat}

Let  $\clC$ be a category with finite limits. An internal category $X$ in $\clC$ is a diagram in $\clC$
\begin{equation}\label{eq-nint-cat}
\xymatrix{
\tens{X_1}{X_0} \ar^(0.65){m}[r] & X_1 \ar^{d_0}[r]<2.5ex> \ar^{d_1}[r] & X_0
\ar^{s}[l]<2ex>
}
\end{equation}
where $m,d_0,d_1,s$ satisfy the usual axioms of a category (see for instance \cite{Borc} for details). An internal functor is a morphism of diagrams like \eqref{eq-nint-cat} commuting in the obvious way. We denote by $\Cat \clC$ the category of internal categories and internal functors.

The category $\cat{n}(\clC)$ of \nfol categories in $\clC$ is defined inductively by iterating $n$ times the internal category construction. That is, $\cat{1}(\clC)=\Cat$ and, for $n>1$,
\begin{equation*}
  \cat{n}(\clC)= \Cat(\cat{n-1}(\clC)).
\end{equation*}

When $\clC=\Set$, $\cat{n}(\Set)$ is simply denoted by $\cat{n}$ and called the category of \nfol categories (double categories when $n=2$).

\subsection{Nerve functors}\label{sus-ner-funct}

There is a nerve functor
\begin{equation*}
    N:\Cat\clC \rw \funcat{}{\clC}
\end{equation*}
such that, for $X\in\Cat\clC$
\begin{equation*}
    (N X)_k=
    \left\{
      \begin{array}{ll}
        X_0, & \hbox{$k=0$;} \\
        X_1, & \hbox{$k=1$;} \\
        \pro{X_1}{X_0}{k}, & \hbox{$k>1$.}
      \end{array}
    \right.
\end{equation*}
When no ambiguity arises, we shall sometimes denote $(NX)_k$ by $X_k$ for all $k\geq 0$.

The following fact is well known:
\begin{proposition}\label{pro-ner-int-cat}
    A simplicial object in $\clC$ is the nerve of an internal category in $\clC$ if and only if all the Segal maps are isomorphisms.
\end{proposition}

By iterating the nerve construction, we obtain the multinerve functor
\begin{equation*}
    \Nn :\cat{n}(\clC)\rw \funcat{n}{\clC}\;.
\end{equation*}
\begin{definition}\label{def-discrete-nfold}
An internal $n$-fold category $X\in \cat{n}(\clC)$ is said to be discrete if $\Nn X$ is a constant functor.
\end{definition}

 Each object of $\cat{n}(\clC)$ can be considered as an internal category in $\cat{n-1}(\clC)$ in $n$ possible ways, corresponding to the $n$ simplicial directions of its multinerve. To prove this, we use the following lemma, which is a straightforward consequence of the definitions.

\begin{lemma}\label{lem-multin-iff}\
\begin{itemize}
      \item [a)] $X\in\funcat{n}{\clC}$ is the multinerve of an \nfol category in $\clC$ if and only if, for all $1\leq r\leq n$ and $[p_1],\ldots,[p_r]\in\Dop$, $p_r\geq 2$
\begin{equation}\label{eq-multin-iff}
\begin{split}
    &  X(p_1,...,p_r,\mi)\cong\\
    &\resizebox{1.0\hsize}{!}{$\cong\pro{X(p_1,...,p_{r-1},1,\mi)}{X(p_1,...,p_{r-1},0,\mi)}{p_r}$}
\end{split}
\end{equation}
      \item [b)] Let $X\in\cat{n}(\clC)$. For each $1\leq k\leq n$, $[i]\in\Dop$, there is $X_i\up{k}\in\cat{n-1}(\clC)$ with
\begin{equation*}
    \Nb{n-1}X_i\up{k} (p_1,\ldots,p_{n-1})=\Nn X(p_1,\ldots,p_{k-1},i,p_k,\ldots,p_{n-1})
\end{equation*}
    \end{itemize}
\end{lemma}
\begin{proof}\

  a) By induction on $n$. By Proposition \ref{pro-ner-int-cat}, it is true for $n=1$. Suppose it holds for $n-1$ and let $X\in\Cat(\cat{n-1}(\clC))$ with objects of objects (resp. arrows) $X_0$ (resp. $X_1$); denote $(NX)_p=X_p$. By definition of the multinerve
      \begin{equation*}
        (\Nb{n} X)(p_1,\ldots,p_r,\mi)=\Nb{n-1}X_{p_1}(p_2,\ldots,p_r,\mi)\;.
      \end{equation*}
      Hence using the induction hypothesis
\begin{align*}
       &\Nb{n}X(p_1...p_r\,\mi)=\Nb{n-1}X_{p_1}(p_2... p_r\,\mi)\cong\\
&\resizebox{1.0\hsize}{!}{$
\cong \pro{\Nb{n-1} X_{p_1}(p_2... p_{r-1}\,1\,\mi)}
         {\Nb{n-1} X_{p_1}(p_2... p_{r-1}\,0\,\mi)}{p_r}=$}\\
&\resizebox{1.0\hsize}{!}{
          $=\pro{\Nb{n} X(p_1... p_{r-1}\,1\,\mi)}
         {\Nb{n} X(p_1... p_{r-1}\,0\,\mi)}{p_r}.$}
\end{align*}
Conversely, suppose $X\in\funcat{n}{\clC}$ satisfies \eqref{eq-multin-iff}. Then for each $[p]\in\Dop$, $X(p,\mi)$ satisfies \eqref{eq-multin-iff}, hence
\begin{equation*}
    X(p,\mi)=\Nb{n-1}X_p
\end{equation*}
for $X_p\in\cat{n-1}(\clC)$. Also, by induction hypothesis
\begin{equation*}
    X(p,\mi)=\pro{X(1,\mi)}{X(0,\mi)}{p}\;.
\end{equation*}
Thus we have the object $X\in\cat{n}(\clC)$ with objects $X_0$, arrows $X_1$ and $X_p=X(p,\mi)$ as above.

\bigskip
b) By part a), there is an isomorphism for $p_r\geq 2$

\begin{flalign*}
       &\Nb{n}X(p_1...p_n)=&
\end{flalign*}
\begin{equation*}
\resizebox{1.0\hsize}{!}{$\pro{\Nb{n}X(p_1...p_{r-1}\, 1 ...p_n)}{\Nb{n}X(p_1...p_{r-1}\, 0 ...p_n)}{p_r}$}\;.
\end{equation*}
In particular, evaluating this at $p_k=i$, this is saying the $(n-1)$-simplicial group taking $(p_1...p_n)$ to $\Nb{n}X(p_1...p_{k-1}\, i ...p_{n-1})$ satisfies condition \eqref{eq-multin-iff} in part a). Hence by part a) there exists $X_i\up{k}$ with
\begin{equation*}
    \Nb{n-1}X_i\up{k}(p_1...p_{n-1})=\Nb{n}X(p_1...p_{k-1}\, i ...p_{n-1})
\end{equation*}
as required.
\end{proof}
\begin{proposition}\label{pro-assoc-iso}
    For each $1\leq k\leq n$ there is an isomorphism $\xi_k:\cat{n}(\clC)\rw \cat{n}(\clC)$ which associates to $X=\cat{n}(\clC)$ an object $\xi_k X$ of $\Cat(\cat{n-1}(\clC))$ with
    \begin{equation*}
        (\xi_k X)_i=X_i\up{k}\qquad i=0,1
    \end{equation*}
    with $X_i\up{k}$ as in Lemma \ref{lem-multin-iff}.
\end{proposition}
\begin{proof}
Consider the object of $\funcat{}{\funcat{n-1}{\clC}}$ taking $i$ to the $\nm$-simplicial object associating to $(\seqc{p}{1}{n-1})$ the object
\begin{equation*}
    \Nb{n}X(p_1...p_{k-1}\, i \, p_{k+1}...p_{n-1})\;.
\end{equation*}
By Lemma \ref{lem-multin-iff} b), the latter is the multinerve of $X_i\up{k}\in \cat{n-1}(\clC)$. Further, by Lemma \ref{lem-multin-iff} a), we have
\begin{equation*}
    \Nb{n-1}X_i\up{k}\cong \pro{\Nb{n-1}X_1\up{k}}{\Nb{n-1}X_0\up{k}}{i}\;.
\end{equation*}
Hence $\Nb{n}X$ as a simplicial object in $\funcat{n-1}{\clC}$ along the $k^{th}$ direction, has
\begin{equation*}
    (\Nb{n}X)_i=
    \left\{
      \begin{array}{ll}
        \Nb{n-1}X_i\up{k}, & \hbox{$i=0,1$;} \\
        \Nb{n-1}(\pro{X_1\up{k}}{X_0\up{k}}{i}), & \hbox{for $i\geq 2$.}
      \end{array}
    \right.
\end{equation*}
This defines $\xi_k X\in\Cat(\cat{n-1}(\clC))$ with
\begin{equation*}
    (\xi_k X)_i=\Nb{n-1}X_i\up{k}\qquad i=0,1\;.
\end{equation*}
We now define the inverse for $\xi_k$. Let $X\in\Cat(\cat{n-1}(\clC))$, and let $X_i=\pro{X_1}{X_0}{i}$ for $i\geq 2$. The $n$-simplicial object $X_{k}$ taking $(p_1,\ldots,p_n)$ to
\begin{equation*}
    \Nb{n}X_{p_k}(p_1...p_{k-1}p_{k+1}...p_n)
\end{equation*}
satisfies condition \eqref{eq-multin-iff}, as easily seen. Hence by Lemma \ref{lem-multin-iff} there is $\xi'_k X\in\cat{n}\clC$ such that $\Nb{n }\xi'_k X=X_k$. It is immediate to check that $\xi_k$ and $\xi'_k$ are inverse bijections.
\end{proof}
\begin{definition}\label{def-ner-func-dirk}
    The nerve functor in the $k^{th}$ direction is defined as the composite
    \begin{equation*}
        \Nu{k}:\cat{n}(\clC)\xrw{\xi_k}\Cat(\cat{n-1}(\clC))\xrw{N}\funcat{}{\cat{n-1}(\clC)}
    \end{equation*}
    so that, in the above notation,
    \begin{equation*}
        (\Nu{k}X)_i=X\up{k}_i\qquad i=0,1\;.
    \end{equation*}
    Note that $\Nb{n}=\Nu{n}...\Nu{2}\Nu{1}$.
\end{definition}
\begin{notation}\label{not-ner-func-dirk}
    When $\clC=\Set$ we shall denote
    \begin{equation*}
        J_n=\Nu{n-1}\ldots \Nu{1}:\cat{n}\rw\funcat{n-1}{\Cat}\;.
    \end{equation*}
\end{notation}
 Thus $J_n$ amounts to taking the nerve construction in all but the last simplicial direction. The functor $J_n$ is fully faithful, thus we can identify $\cat{n}$ with the image $J_n(\cat{n})$ of the functor $J_n$.

 Given $X\in\cat{n}$, when no ambiguity arises we shall denote, for each $(s_1,\ldots ,s_{n-1})\in\Dmenop$
\begin{equation*}
    X_{s_1,\ldots ,s_{n-1}}=(J_n X)_{s_1,\ldots ,s_{n-1}}\in\Cat
\end{equation*}
and more generally, if $1\leq j \leq n-1$,
\begin{equation*}
    X_{s_1,\ldots ,s_{j}}=(\Nu{j}\ldots \Nu{1} X)_{s_1,\ldots ,s_{j}}\in\cat{n-j}\;.
\end{equation*}
Let $ob : \Cat \clC \rw \clC$ be the object of object functor. The left adjoint to $ob$ is the discrete internal category functor $d$. By Proposition \ref{pro-assoc-iso}  we then have
\begin{equation*}
\xymatrix{
    \cat{n}\clC \oset{\xi_n}{\cong}\Cat(\cat{n-1}\clC) \ar@<1ex>[r]^(0.65){ob} & \cat{n-1}\clC \ar@<1ex>[l]^(0.35){d}\;.
}
\end{equation*}

We denote
\begin{equation*}
\di{n}=\xi^{-1}_{n}\circ d \text{\;\;for\;\;} n>1,\; \di{1}=d \;.
\end{equation*}
Thus $\di{n}$ is the discrete inclusion of $\cat{n-1}\clC$ into $\cat{n}\clC$ in the $n^{th}$ direction.

\bk

The following is a characterization of objects of $\funcat{n-1}{\Cat}$ in the image of the functor $J_n$ in \ref{not-ner-func-dirk}.
\begin{lemma}\label{lem-char-obj}
     Let $L\in \funcat{n-1}{\Cat}$ be such that, for all $\uk\in\dop{n-1}$, $1\leq i\leq n-1$ and $k_i\geq 2$, the Segal maps are isomorphisms:
     \begin{equation}\label{eq-lem-char-obj}
        L_{\uk} \cong\pro{L_{\uk(1,i)}}{L_{\uk(0,i)}}{k_i}\;.
     \end{equation}
     Then $L\in\cat{n}$.
\end{lemma}
\begin{proof}
By induction on $n$. When $n=2$, $L\in\funcat{}{\Cat}$, $k\in \Dop$, $i=1$, $\uk(1,i)=1$, $\uk(0,i)=0$, $k_i=k=2$ and
\begin{equation*}
    L_k\cong\pro{L_1}{L_0}{k}\;.
\end{equation*}
Thus by Proposition \ref{pro-ner-int-cat}, $L\in\cat{2}$.

Suppose the lemma holds for $(n-1)$ and let $L \in \funcat{n-1}{\Cat}$ be as in the hypothesis. Consider $L_j\in\funcat{n-2}{\Cat}$ for $j\geq 0$. Let $\ur\in\dop{n-2}$ and denote $\uk=(j,\ur)\in\dop{n-1}$. Then, for any $2\leq i\leq n-1$, $k_i=r_{i-1}$ and
\begin{equation*}
    L_{\uk}=(L_j)_{\ur},\quad L_{\uk(1,i)}=(L_j)_{\ur(1,i-1)},\quad L_{\uk(0,i)}=(L_j)_{\ur(0,i-1)}\;.
\end{equation*}
Therefore \eqref{eq-lem-char-obj} implies
\begin{equation*}
    (L_j)_{\ur}=\pro{(L_j)_{\ur(1,i-1)}}{(L_j)_{\ur(0,i-1)}}{r_{i-1}}\;.
\end{equation*}
This means that $L_j$ satisfies the inductive hypothesis and therefore $L_j\in\cat{n-1}$.
Taking $i=1$ in \eqref{eq-lem-char-obj} we see that, for each $k_1\geq 2$ and $\ur=(k_2,\ldots,k_{n-1})\in\dop{n-2}$,
\begin{equation*}
    (L_{k_1})_{\ur}=\pro{(L_1)_{\ur}}{(L_0)_{\ur}}{k_1}\;.
\end{equation*}
That is, we have isomorphisms in $\cat{n-1}$
\begin{equation*}
    L_{k_1}\cong \pro{L_1}{L_0}{k_1}\;.
\end{equation*}
We conclude from Proposition \ref{pro-ner-int-cat} that $L\in\cat{n}$.

\end{proof}
\begin{lemma}\label{lem-char-obj-II}
    Let $P$ be the pullback in $\funcat{n-1}{\Cat}$ of the diagram in $A\rw C \lw B$. Suppose that for each $k\geq 2$ there are isomorphisms of Segal maps in $\funcat{n-2}{\Cat}$
    \begin{equation*}
        A_k\cong \pro{A_1}{A_0}{k},\;\;  C_k\cong \pro{C_1}{C_0}{k},\;\;  B_k\cong \pro{B_1}{B_0}{k}\;.
    \end{equation*}
    Then $P_k\cong \pro{P_1}{P_0}{k}$.
\end{lemma}
\begin{proof}
We show this for $k=2$, the case $k>2$ being similar. Since the nerve functor $N:\Cat\rw\funcat{}{\Set}$ commutes with pullbacks (as it is right adjoint) and pullbacks in $\funcat{n-1}{\Cat}$ are computed pointwise, for each $\us\in\dop{n-1}$ we have a pullback in $\Set$
\begin{equation*}
    \xymatrix{
    (NP)_{2\us} \ar[r] \ar[d] & (NA)_{2\us} \ar[d]\\
    (NC)_{2\us} \ar[r] & (NC)_{2\us}
    }
\end{equation*}
where
\begin{equation*}
    (NA)_{2\us}=\tens{(NA)_{1\us}}{(NA)_{0\us}}
\end{equation*}
and similarly for $NC$ and $NB$. We then calculate
\begin{align*}
    &(NP)_{2\us}=(NA)_{2\us} \tiund{(NC)_{2\us}} (NB)_{2\us} =\\
    & \resizebox{1.0\hsize}{!}{$
    =\{ \tens{(NA)_{1\us}}{(NA)_{0\us}} \}\tiund{\tens{(NC)_{1\us}}{(NC)_{0\us}}} \{ \tens{(NB)_{1\us}}{(NB)_{0\us}}\}\cong $}\\
    & \resizebox{1.0\hsize}{!}{$
    \cong\{ \tens{(NB)_{1\us}}{(NC)_{0\us}} \}\tiund{\tens{(NA)_{1\us}}{(NC)_{0\us}}} \{ \tens{(NB)_{1\us}}{(NC)_{0\us}}\}= $}\\
    &= \tens{(NP)_{1\us}}{(NP)_{0\us}}\;.
\end{align*}
In the above, the isomorphism before the last takes $(x_1,x_2,x_3,x_4)$ to $(x_1,x_3,x_2,x_4)$. Since this holds for all $\us$, it follows that
\begin{equation*}
    P_2\cong \tens{P_1}{P_0}\;.
\end{equation*}
The case $k>2$ is similar.

\end{proof}

\subsection{Some functors on $ \Cat$}\label{sbs-funct-cat}

The connected component functor
\begin{equation*}
    q: \Cat\rw \Set
\end{equation*}
associates to a category its set of paths components. This is left adjoint to the discrete category functor
\begin{equation*}
    \di{1}:\Set \rw \Cat
\end{equation*}
associating to a set $X$ the discrete category on that set. We denote by
\begin{equation*}
    \zgu{1}:\Id\Rw \di{1}q
\end{equation*}
the unit of the adjunction $q\dashv \di{1}$.
\begin{lemma}\label{lem-q-pres-fib-pro}
    $q$ preserves fiber products over discrete objects and sends
    equivalences of categories to isomorphisms.
\end{lemma}
\begin{proof}
We claim that $q$ preserves products; that is, given categories
$\clC$ and $\clD$, there is a bijection
\begin{equation*}
    q(\clC\times \clD)=q(\clC)\times q(\clD)\;.
\end{equation*}
In fact, given $(c,d)\in q(\clC\times \clD)$ the map
$q(\clC\times\clD)\rw q(\clC)\times q(\clD)$ given by
$[(c,d)]=([c],[d])$ is well defined and is clearly surjective. On
the other hand, this map is also injective: given $[(c,d)]$ and
$[(c',d')]$ with $[c]=[c']$ and $[d]=[d']$, we have paths in $\clC$
\newcommand{\lin}{-\!\!\!-\!\!\!-}
\begin{equation*}
\xymatrix @R5pt{c \hspace{2mm} \lin \hspace{2mm}\cdots \hspace{2mm}
\lin
\hspace{2mm}c'\\
d \hspace{2mm} \lin \hspace{2mm}\cdots \hspace{2mm} \lin
\hspace{2mm}d' }
\end{equation*}
and hence a path in $\clC\times \clD$
\begin{equation*}
(c,d)\hspace{2mm}\lin\hspace{2mm}\cdots\hspace{2mm}\lin\hspace{2mm}(c',d)
\hspace{2mm}\lin\hspace{2mm}\cdots\hspace{2mm}\lin\hspace{2mm}(c',d')\;.
\end{equation*}
Thus $[(c,d)]=[(c',d')]$ and so the map is also injective, hence it
is a bijection, as claimed.

Given a diagram in $\Cat$ $\xymatrix{\clC\ar_{f}[r] & \clE & \clD
\ar^{g}[l]}$ with $\clE$ discrete, we have
\begin{equation}\label{eq-q-pres-fib-pro}
    \clC\tiund{\clE}\clD=\underset{x\in\clE}{\coprod}\clC_x\times
    \clD_x
\end{equation}
where $\clC_x,\;\clD_x$ are the full subcategories of $\clC$ and
$\clD$ with objects $c,\;d$ such that \;$f(c)=x=g(d)$. Since $q$ preserves
products and (being left adjoint) coproducts, we conclude by
\eqref{eq-q-pres-fib-pro} that
\begin{equation*}
    q(\clC\tiund{\clE}\clD)\cong q(\clC)\tiund{\clE}\,q(\clD)\;.
\end{equation*}
Finally, if $F:\clC\simeq \clD:G$ is an equivalence of categories,
$FG\,\clC\cong\clC$ and $FG\,\clD\cong \clD$ which implies that
$qF\,qG\,\clC\cong q\clC$ and $qF\,qG\,\clD\cong q\clD$, so $q\clC$
and $q\clD$ are isomorphic.
\end{proof}
The isomorphism classes of objects functot
\begin{equation*}
    p:\Cat\rw\Set
\end{equation*}
associates to a category the set of isomorphism classes of its objects. Notice that if $\clC$ is a groupoid, $p\clC=q\clC$.
\subsection{Pseudo-functors and their strictification}\label{sbs-pseudo-functors}
We recall the classical theory of strictification of pseudo-functors, see \cite{PW}, \cite{Lack}.

The functor 2-category $\funcat{n}{\Cat}$ is 2-monadic over $[ob(\Dnop),\Cat]$ where $ob(\Dnop)$ is the set of objects of $\Dnop$. Let
\begin{equation*}
    U:\funcat{n}{\Cat}\rw [ob(\Dnop),\Cat]
\end{equation*}
be the forgetful functor $(UX)_{\uk}=X_{\uk}$. Its left adjoint $F$ is given on objects by
\begin{equation*}
    (FH)_{\uk}=\underset{\ur\in ob(\Dnop)}{\coprod}\Dnop(\ur,\uk)\times H_{\ur}
\end{equation*}
for $H\in [ob(\Dmenop),\Cat]$, $\uk\in ob(\Dmenop)$. If $T$ is the monad corresponding to the adjunction $F\dashv U$, then
\begin{equation*}
    (TH)_{\uk}=\underset{\ur\in ob(\Dnop)}{\coprod}\Dnop(\ur,\uk)\times H_{\ur}
\end{equation*}
A pseudo $T$-algebra is given by $H\in [ob(\Dnop),\Cat]$,
\begin{equation*}
    h_{n}: \underset{\ur\in ob(\Dnop)}{\coprod}\Dnop(\ur,\uk)\times H_{\ur} \rw H_{\uk}
\end{equation*}
and additional data, as described in \cite{PW}. This amounts precisely to functors from $\Dnop$ to $\Cat$ and the 2-category $\sf{Ps\mi T\mi alg}$ of pseudo $T$-algebras corresponds to the 2-category $\Ps\funcat{n}{\Cat}$ of pseudo-functors, pseudo-natural transformations and modifications.

The strictification result proved in \cite{PW} yields that every pseudo-functor from $\Dnop$ to $\Cat$ is equivalent, in $\Ps\funcat{n}{\Cat}$, to a 2-functor.

Given a pseudo $T$-algebra as above, \cite{PW} consider the factorization  of $h:TH\rw H$ as
\begin{equation*}
    TH\xrw{v}L\xrw{g}H
\end{equation*}
with $v_{\uk}$ bijective on objects and $g_{\uk}$ fully faithful, for each $\uk\in\Dnop$. It is shown in \cite{PW} that it is possible to give a strict $T$-algebra structure $TL\rw L$ such that $(g,Tg)$ is an equivalence of pseudo $T$-algebras. It is immediate to see that, for each $\uk\in\Dnop$, $g_{\uk}$ is an equivalence of categories.

Further, it is shown in \cite{Lack} that $\St:\psc{n}{\Cat}\rw\funcat{n}{\Cat}$ as described above is left adjoint to the inclusion
\begin{equation*}
  J:\funcat{n}{\Cat}\rw\psc{n}{\Cat}
\end{equation*}
 and that the components of the units are equivalences in $\psc{n}{\Cat}$.
\subsection{Transport of structure}\label{transport-structure} We now recall a general categorical technique, known as transport of structure along an adjunction, with one of its applications. This will be used crucially in the proof of Theorem \ref{the-XXXX}.
\begin{theorem}\rm{\cite[Theorem 6.1]{lk}}\em\label{s2.the1}
    Given an equivalence $\;\eta,\;\zve : f \dashv f^* : A\rw B$ in the complete and locally small
    2-category $\clA$, and an algebra $(A,a)$ for the monad $T=(T,i,m)$ on $\clA$, the
    equivalence enriches to an equivalence
\begin{equation*}
  \eta,\zve:(f,\ovll{f})\vdash (f^*,\ovll{f^*}):(A,a)\rw(B,b,\hat{b},\ovl{b})
\end{equation*}
in $\PsTalg$, where $\hat{b}=\eta$, $\;\ovl{b}=f^* a\cdot T\zve \cdot
Ta\cdot T^2 f$, $\;\ovll{f}=\zve^{-1} a\cdot Tf$, $\;\ovll{f^*}=f^* a\cdot
T\zve$.
\end{theorem}
Let $\eta',\zve':f'\vdash f'^{*}:A'\rw B'$ be another equivalence in $\clA$ and
let $(B',b',\hat{b'},\ovl{b'})$ be the corresponding pseudo-$T$-algebra as in
Theorem \ref{s2.the1}. Suppose $g:(A,a)\rw(A',a')$ is a morphism in $\clA$ and
$\gamma$ is an invertible 2-cell in $\clA$
\begin{equation*}
  \xy
    0;/r.10pc/:
    (-20,20)*+{B}="1";
    (-20,-20)*+{B'}="2";
    (20,20)*+{A}="3";
    (20,-20)*+{A'}="4";
    {\ar_{f^*}"3";"1"};
    {\ar_{h}"1";"2"};
    {\ar^{f'^*}"4";"2"};
    {\ar^{g}"3";"4"};
    {\ar@{=>}^{\gamma}(0,3);(0,-3)};
\endxy
\end{equation*}
Let $\ovl{\gamma}$ be the invertible 2-cell given by the following pasting:
\begin{equation*}
    \xy
    0;/r.15pc/:
    (-40,40)*+{TB}="1";
    (40,40)*+{TB'}="2";
    (-40,-40)*+{B}="3";
    (40,-40)*+{B'}="4";
    (-20,20)*+{TA}="5";
    (20,20)*+{TA'}="6";
    (-20,-20)*+{A}="7";
    (20,-20)*+{A'}="8";
    {\ar^{Th}"1";"2"};
    {\ar_{b}"1";"3"};
    {\ar^{b'}"2";"4"};
    {\ar_{h}"3";"4"};
    {\ar_{Tg}"5";"6"};
    {\ar^{}"5";"7"};
    {\ar_{}"6";"8"};
    {\ar^{g}"7";"8"};
    {\ar_{Tf^*}"5";"1"};
    {\ar^{f^*}"7";"3"};
    {\ar^{Tf'^*}"6";"2"};
    {\ar^{f'^*}"8";"4"};
    {\ar@{=>}^{(T\gamma)^{-1}}(0,33);(0,27)};
    {\ar@{=>}^{\gamma}(0,-27);(0,-33)};
    {\ar@{=>}^{\ovll{f'^*}}(30,3);(30,-3)};
    {\ar@{=>}^{\ovll{f^*}}(-30,3);(-30,-3)};
\endxy
\end{equation*}
Then it is not difficult to show that
$(h,\ovl{\gamma}):(B,b,\hat{b},\ovl{b})\rw(B',b',\hat{b'},\ovl{b'})$ is a
pseudo-$T$-algebra morphism.

The following fact is essentially known and, as sketched in the proof below, it is an instance of Theorem \ref{s2.the1}
\begin{lemma}\cite{PP}\label{lem-PP}
     Let $\clC$ be a small 2-category, $F,F':\clC\rw\Cat$ be 2-functors, $\alpha:F\rw F'$
    a 2-natural transformation. Suppose that, for all objects $C$ of $\clC$, the
    following conditions hold:
\begin{itemize}
  \item [i)] $G(C),\;G'(C)$ are objects of $\Cat$ and there are adjoint equivalences of
  categories $\mu_C\vdash\eta_C$, $\mu'_C\vdash\eta'_C$,
\begin{equation*}
  \mu_C:G(C)\;\rightleftarrows\;F(C):\eta_C\qquad\qquad
  \mu'_C:G'(C)\;\rightleftarrows\;F'(C):\eta'_C,
\end{equation*}
  \item [ii)] there are functors $\beta_C:G(C)\rw G'(C),$
  \item [iii)] there is an invertible 2-cell
\begin{equation*}
  \gamma_C:\beta_C\,\eta_C\Rightarrow\eta'_C\,\alpha_C.
\end{equation*}
\end{itemize}
Then
\begin{itemize}
  \item [a)] There exists a pseudo-functor $G:\clC\rw\Cat$ given on objects by $G(C)$,
  and pseudo-natural transformations $\eta:F\rw G$, $\mu:G\rw F$ with
  $\eta(C)=\eta_C$, $\mu(C)=\mu_C$; these are part of an adjoint equivalence
  $\mu\vdash\eta$ in the 2-category $\Ps[\clC,\Cat]$.
  \item [b)] There is a pseudo-natural transformation $\beta:G\rw G'$ with
  $\beta(C)=\beta_C$ and an invertible 2-cell in $\Ps[\clC,\Cat]$,
  $\gamma:\beta\eta\Rightarrow\eta\alpha$ with $\gamma(C)=\gamma_C$.
\end{itemize}
\end{lemma}
\begin{proof}
Recall \cite{PW} that the functor 2-category $[\clC,\Cat]$ is 2-monadic over
$[ob(\clC),\Cat]$, where $ob(\clC)$ is the set of objects in $\clC$. Let
\begin{equation*}
  \clU:[\clC,\Cat]\rw[ob(\clC),\Cat]
\end{equation*}
be the forgetful functor. Let $T$ be the 2-monad; then the pseudo-$T$-algebras are precisely the pseudo-functors from
$\clC$ to $\Cat$.

Then the adjoint equivalences $\mu_C\vdash\eta_C$ amount precisely to an
adjoint equivalence in $[ob(\clC),\Cat]$, $\;\mu_0\vdash\eta_0$,
$\;\mu_0:G_0\;\;\rightleftarrows\;\;\clU F:\eta_0$ where $\;G_0(C)=G(C)$ for
all $C\in ob(\clC)$. This equivalence enriches to an
adjoint equivalence $\mu\vdash\eta$ in $\Ps[\clC,\Cat]$
\begin{equation*}
  \mu:G\;\rightleftarrows\; F:\eta
\end{equation*}
between $F$ and a pseudo-functor $G$; it is $\clU G=G_0$, $\;\clU\eta=\eta_0$,
$\;\clU\mu=\mu_0$; hence on objects $G$ is given by $G(C)$, and
$\eta(C)=\clU\eta(C)=\eta_C$, $\;\mu(C)=\clU\mu(C)=\mu_C$.

Let $\nu_C:\Id_{G(C)}\Rw\eta_C\mu_C$ and $\zve_C:\mu_C\eta_C\Rw\Id_{F(C)}$ be
the unit and counit of the adjunction $\mu_C\vdash\eta_C$. Given a morphism $f:C\rw D$ in $\clC$, it is
\begin{equation*}
  G(f)=\eta_D F(f)\mu_C
\end{equation*}
and we have natural isomorphisms:
\begin{align*}
   & \eta_f : G(f)\eta_C=\eta_D F(f)\mu_C\eta_C\overset{\eta_D F(f)\zve_C}{=\!=\!=\!=\!\Rw} \eta_D F(f)\\
   & \mu_f : F(f)\mu_C\overset{\nu_{F(f)}\mu_C}{=\!=\!=\!\Rw}\mu_D\eta_D F(f)\mu_C=\mu_D
   G(f).
\end{align*}
Also, the natural isomorphism
\begin{equation*}
  \beta_f: G'(f)\beta_C\Rw\beta_D G(f)
\end{equation*}
is the result of the following pasting
\begin{equation*}
    \xy
    0;/r.15pc/:
    (-40,40)*+{G(C)}="1";
    (40,40)*+{G'(C)}="2";
    (-40,-40)*+{G(D)}="3";
    (40,-40)*+{G'(D)}="4";
    (-20,20)*+{F(C)}="5";
    (20,20)*+{F'(C)}="6";
    (-20,-20)*+{F(D)}="7";
    (20,-20)*+{F'(D)}="8";
    {\ar^{\beta_C}"1";"2"};
    {\ar_{G(f)}"1";"3"};
    {\ar^{G'(f)}"2";"4"};
    {\ar_{\beta_D}"3";"4"};
    {\ar^{\alpha_C}"5";"6"};
    {\ar^{F(f)}"5";"7"};
    {\ar_{F'(f)}"6";"8"};
    {\ar_{\alpha'_D}"7";"8"};
    {\ar^{}"5";"1"};
    {\ar^{}"7";"3"};
    {\ar^{}"6";"2"};
    {\ar^{}"8";"4"};
    {\ar@{=>}^{\gamma_C}(0,33);(0,27)};
    {\ar@{=>}^{\gamma_D^{-1}}(0,-27);(0,-33)};
    {\ar@{=>}^{\eta_f'}(30,3);(30,-3)};
    {\ar@{=>}^{\eta_f}(-30,3);(-30,-3)};
\endxy
\end{equation*}
\end{proof}
\section{The three Segal-type models}\label{sec-segal-type-model}
\subsection{Homotopically discrete $n$-fold categories}\label{subs-hom-disc-n-fold}
We first recall the category of homotopically discrete \nfol categories, introduced by the author in \cite{Pa1}. This is needed to define weakly globular $n$-fold categories.
\begin{definition}\label{def-hom-dis-ncat}
    Define inductively the full subcategory $\cathd{n}\subset\cat{n}$ of homotopically discrete \nfol categories.

    For $n=1$, $\cathd{1}=\cathd{}$ is the category of  equivalence relations. Denote by $\p{1}=p:\Cat\rw\Set$ the isomorphism classes of objects functor.

    Suppose, inductively, that for each $1\leq k\leq n-1$ we defined $\cathd{k}\subset\cat{k}$ and $k$-equivalences such that the following holds:
    \begin{itemize}
      \item [a)] The $k^{th}$ direction in $\cathd{k}$ is groupoidal; that is, if $X\in\cathd{k}$, $\xi_{k}X\in\Gpd(\cat{k-1})$ (where $\xi_{k}X$ is as in Proposition \ref{pro-assoc-iso}).

      \item [b)] There is a functor $\p{k}:\cathd{k}\rw\cathd{k-1}$ making the following diagram commute:
          \begin{equation}\label{eq1-p-fun-def}
            \xymatrix{
            \cathd{k} \ar^{\Nu{k-1}...\Nu{1}}[rrr] \ar_{p^{(k)}}[d] &&& \funcat{k-1}{\Cat} \ar^{\bar p}[d]\\
            \cathd{k-1} \ar_{\N{k-1}}[rrr] &&& \funcat{k-1}{\Set}
            }
          \end{equation}
          Note that this implies that $(\p{k}X)_{s_1 ... s_{k-1}}=p X_{s_1 ... s_{k-1}}$ for all $(s_1 ... s_{k-1})\in\dop{k-1}$.

    \end{itemize}

   $\cathd{n}$ is the full subcategory of $\funcat{}{\cathd{n-1}}$ whose objects $X$ are such that
    \bigskip
    \begin{itemize}
      \item [(i)]  $\hspace{30mm} X_s\cong\pro{X_1}{X_0}{s} \quad \mbox{for all} \; s \geq 2.$

\medskip
    In particular this implies that $X\in \Cat(\Gpd(\cat{n-2})) =\Gpd(\cat{n-1})$ and the $n^{th}$ direction in $X$ is groupoidal.
    \medskip
      \item [(ii)] The functor
      \begin{equation*}
        \op{n-1}:\cathd{n}\subset \funcat{}{\cathd{n-1}}\rw\funcat{}{\cathd{n-2}}
      \end{equation*}
      restricts to a functor
      \begin{equation*}
        \p{n}:\cathd{n}\rw\cathd{n-1}
      \end{equation*}
     Note that this implies that $(\p{n}X)_{s_1 ... s_{n-1}}=p X_{s_1 ... s_{n-1}}$ and that the following diagram commutes
          \begin{equation}\label{eq2-p-fun-def}
            \xymatrix{
            \cathd{n} \ar^{\Nu{n-1}...\Nu{1}}[rrr] \ar_{p^{(n)}}[d] &&& \funcat{n-1}{\Cat} \ar^{\bar p}[d]\\
            \cathd{n-1} \ar_{\N{n-1}}[rrr] &&& \funcat{n-1}{\Set}
            }
          \end{equation}
     \end{itemize}
\end{definition}
\mk
\begin{definition}\label{def-hom-dis-ncat-0}

    Denote by $\zgu{n}_X:X\rw \di{n}\p{n}X$ the morphism given by
    \begin{equation*}
        (\zgu{n}_X)_{s_1...s_{n-1}} :X_{s_1...s_{n-1}} \rw d p X_{s_1...s_{n-1}}
    \end{equation*}
    for all  $(s_1,...,s_{n-1})\in \dop{n-1}$. Denote by
    \begin{equation*}
        X^d =\di{n}\di{n-1}...\di{1}\p{1}\p{2}...\p{n}X
    \end{equation*}
    and by $\zg\lo{n}$ the composite
    \begin{equation*}
        X\xrw{\zgu{n}}\di{n}\p{n}X \xrw{\di{n}\zgu{n-1}} \di{n}\di{n-1}\p{n-1}\p{n}X \rw \cdots \rw X^d\;.
    \end{equation*}
    For each $a,b\in X_0^d$ denote by $X(a,b)$ the fiber at $(a,b)$ of the map
    \begin{equation*}
        X_1 \xrw{(d_0,d_1)} X_0\times X_0 \xrw{\zg\lo{n}\times\zg\lo{n}} X_0^d\times X_0^d\;.
    \end{equation*}
\end{definition}
\begin{definition}\label{def-hom-dis-ncat-1}
Define inductively $n$-equivalences in $\cathd{n}$. For $n=1$, a 1-equivalence is an equivalence of categories. Suppose we defined $\nm$-equivalences in $\cathd{n-1}$. Then a map $f:X\rw Y$ in $\cathd{n}$ is an $n$-equivalence if, for all $a,b \in X_0^d$, $f(a,b):X(a,b) \rw Y(fa,fb)$ and $\p{n}f$ are $\nm$-equivalences.
\end{definition}
\subsection{Weakly globular $n$-fold categories}\label{subs-weak-glob-n-fold}
\begin{definition}\label{def-n-equiv}
    For $n=1$, $\catwg{1}=\Cat$ and $1$-equivalences are equivalences of categories.

    Suppose, inductively, that we defined $\catwg{n-1}$ and $(n-1)$-equivalences. Then $\catwg{n}$ is the full subcategory of $\funcat{}{\catwg{n-1}}$ whose objects $X$ are such that
    \begin{itemize}
      \item [a)] \textsl{Weak globularity condition} $X_0\in\cathd{n-1}$.\mk
      \item [b)] \textsl{Segal condition} For all $k\geq 2$ the Segal maps are isomorphisms:
      \begin{equation*}
        X_k\cong\pro{X_1}{X_0}{k}\;.
      \end{equation*}

      \item [c)] \textsl{Induced Segal condition} For all $k\geq 2$ the induced Segal maps
      \begin{equation*}
        X_k\rw\pro{X_1}{X^d_0}{k}
      \end{equation*}
      (induced by the map $\zg:X_0\rw X_0^d$) are $(n-1)$-equivalences.\mk

      \item [d)] \textsl{Truncation functor} There is a functor $\p{n}:\catwg{n}\rw\catwg{n-1}$ making the following diagram commute
      \begin{equation*}
        \xymatrix{
        \catwg{n} \ar^{J_n}[rr] \ar_{\p{n}}[d] && \funcat{n-1}{\Cat} \ar^{\ovl p}[d]\\
        \catwg{n-1} \ar^{\Nb{n-1}}[rr]  && \funcat{n-1}{\Set}
        }
    \end{equation*}
    \end{itemize}
    Given $a,b\in X_0^d$, denote by $X(a,b)$ the fiber at $(a,b)$ of the map
    \begin{equation*}
         X_1\xrw{(\pt_0,\pt_1)} X_0\times X_0 \xrw{\zg\times \zg}  X^d_0\times X^d_0\;.
    \end{equation*}
    We say that a map $f:X\rw Y$ in $\catwg{n}$ is an $n$-equivalence if
    \begin{itemize}
      \item [i)] For all $a,b\in X_0^d$
      \begin{equation*}
        f(a,b): X(a,b) \rw Y(fa,fb)
      \end{equation*}
      is an $(n-1)$-equivalence.\mk

      \item [ii)] $\p{n}f$ is an $(n-1)$-equivalence.
    \end{itemize}
    This completes the inductive step in the definition of $\catwg{n}$.
\end{definition}
\begin{remark}\label{rem-n-equiv}
    It follows by Definition \ref{def-n-equiv}, Definition \ref{def-hom-dis-ncat} and \cite[Proposition 3.4]{Pa2} that $\cathd{n}\subset \catwg{n}$.
\end{remark}
\subsection{Weakly globular Tamsamani $n$-categories}\label{subs-weak-glob-Tam}
\begin{definition}\rm\label{def-wg-ps-cat}
    We define the category $\tawg{n}$ by induction on $n$. For $n=1$, $\tawg{1}=\Cat$ and 1-equivalences are equivalences of categories. We denote by $\p{1}=p:\Cat\rw \Set$ the isomorphism classes of object functor.

    Suppose, inductively, that we defined for each $1 < k\leq n-1$
    \begin{equation*}
        \xymatrix{\tawg{k}\;\ar@{^{(}->}^(0.4){}[r] & \;\funcat{k-1}{\Cat}}
    \end{equation*}
    and $k$-equivalences in $\tawg{k}$ as well as a functor
    \begin{equation*}
        \p{k}:\tawg{k}\rw \tawg{k-1}
    \end{equation*}
    sending $k$-equivalences to $(k-1)$-equivalences and making the following diagram commute:
    \begin{equation}\label{eq-wg-ps-cat}
    \xymatrix@C=30pt{
    \tawg{k} \ar^{J_{k}}[rr]\ar_{\p{k}}[d] && \funcat{k-1}{\Cat} \ar^{\ovl{p}}[d]\\
    \tawg{k-1}  \ar^{\Nb{k-1}}[rr] && \funcat{k-1}{\Set}
    }
    \end{equation}
    Define $\tawg{n}$ to be the full subcategory of $\funcat{}{\tawg{n-1}}$ whose objects $X$ are such that
        \begin{itemize}
      \item [a)] \textsl{Weak globularity condition} \;$X_0\in\cathd{n-1}$.\bk

      \item [b)] \textsl{Induced Segal maps condition}. For all $s\geq 2$ the induced Segal maps
      \begin{equation*}
       X_s  \rw \pro{X_1}{X^d_0}{s}
      \end{equation*}
      (induced by the map $\zg:X_0\rw X_0^d$) are $(n-1)$-equivalences.
    \end{itemize}
\medskip
To complete the inductive step, we need to define $\p{n}$ and $n$-equivalences.
Note that the functor
\begin{equation*}
    \ovl{\p{n-1}}:\funcat{}{\tawg{n-1}}\rw \funcat{}{\tawg{n-2}}
\end{equation*}
restricts to a functor
\begin{equation*}
    \p{n}:\tawg{n}\rw \tawg{n-1}\;.
\end{equation*}
In fact, by \eqref{eq-wg-ps-cat} $\p{n-1}$ preserves pullbacks over discrete objects so that
\begin{equation*}
  \p{n-1}(\pro{X_1}{X_0^d}{s}) \cong\pro{\p{n-1}X_1}{(\p{n-1}X_0^d)}{s}\; .
\end{equation*}
 Further, $\p{n-1}X_0^d = (\p{n-1}X_0)^d$ and $\p{n-1}$ sends $(n-1)$-equivalences to $(n-2)$-equivalences.

 Therefore, the induced Segal maps for $s\geq 2$
\begin{equation*}
    X_s \rw\pro{X_1}{X_0^d}{s}
\end{equation*}
being $(n-1)$-equivalences, give rise to $(n-2)$-equivalences
\begin{equation*}
     \p{n-1}X_s \rw \pro{\p{n-1}X_1}{(\p{n-1}X_0)^d}{s}\; .
\end{equation*}
This shows that $\p{n}X \in \tawg{n-1}$. It is immediate that \eqref{eq-wg-ps-cat} holds at step $n$.

Given $a,b\in X_0^d$, denote by $X(a,b)$ the fiber at $(a,b)$ of the map
    \begin{equation*}
         X_1\xrw{(\pt_0,\pt_1)} X_0\times X_0 \xrw{\zg\times \zg}  X^d_0\times X^d_0\;.
    \end{equation*}
    We say that a map $f:X\rw Y$ in $\tawg{n}$ is an $n$-equivalence if
    \begin{itemize}
      \item [i)] For all $a,b\in X_0^d$
      \begin{equation*}
        f(a,b): X(a,b) \rw Y(fa,fb)
      \end{equation*}
      is an $(n-1)$-equivalence.\mk

      \item [ii)] $\p{n}f$ is an $(n-1)$-equivalence.
    \end{itemize}
    This completes the inductive step in the definition of $\tawg{n}$.

\end{definition}
\begin{remark}\label{rem-wg-ps-cat}
    It follows by Definition \ref{def-n-equiv} that $\catwg{n}\subset\tawg{n}$.
\end{remark}
\begin{definition}\label{def-x-tawg-disc}
    An object $X\in\tawg{n}$ is called discrete if $\Nb{n-1}X$ is a constant functor.
\end{definition}
\begin{example}\label{ex-tam}
Tamsamani $n$-categories.
\mk

A special case of weakly globular Tamsamani $n$-category occurs when $X\in\tawg{n}$ is such that $X_0$ and $X_{\oset{r}{1...1}0}$ are discrete for all $1 \leq r \leq n-2$. The resulting category is the category $\Tan$ of Tamsamani's $n$-categories. Note that, if $X\in\Tan$ then $X_s \in \ta{n-1}$ for all $n$, the induced Segal maps $\hmu{s}$ coincide with the Segal maps
\begin{equation*}
  \nu_s:X_s\rw \pro{X_1}{X_0}{s}
\end{equation*}
and $\p{n}X\in\ta{n-1}$. Hence this recovers the original definition of Tamsamani's weak $n$-category \cite{Ta}.

\end{example}
\subsection{Segalic pseudo-functors}\label{segalic-pseudo} We now recall the notion of Segalic pseudo-functor from \cite{Pa2}.

Let $H\in\Ps\funcat{n}{\Cat}$ be such that $H_{\uk(0,i)}$ is discrete for all $\uk\in\Dmenop$ and all $i\geq 0$. Then the following diagram commutes, for each $k_i\geq 2$.
\begin{equation*}
    \xy
    0;/r.8pc/:
    (0,0)*+{H_{\uk}}="1";
    (-7,-5)*+{H_{\uk(1,i)}}="2";
    (-2,-5)*+{H_{\uk(1,i)}}="3";
    (9,-5)*+{H_{\uk(1,i)}}="4";
    (-10,-9)*+{H_{\uk(0,i)}}="5";
    (-5,-9)*+{H_{\uk(0,i)}}="6";
    (0,-9)*+{H_{\uk(0,i)}}="7";
    (6,-9)*+{H_{\uk(0,i)}}="8";
    (12,-9)*+{H_{\uk(0,i)}}="9";
    (3,-5)*+{\cdots}="10";
    (3,-9)*+{\cdots}="11";
    {\ar_{\nu_1}"1";"2"};
    {\ar^{\nu_2}"1";"3"};
    {\ar^{\nu_k}"1";"4"};
    {\ar_{d_1}"2";"5"};
    {\ar^{d_0}"2";"6"};
    {\ar^{d_1}"3";"6"};
    {\ar^{d_0}"3";"7"};
    {\ar_{d_1}"4";"8"};
    {\ar^{d_0}"4";"9"};
    \endxy
\end{equation*}
There is therefore a unique Segal map
\begin{equation*}
    H_{\uk}\rw \pro{H_{\uk(1,i)}}{H_{\uk(0,i)}}{k_i}\;.
\end{equation*}
\begin{definition}\label{def-seg-ps-fun}
    We\; define\; the \;subcategory \; $\segpsc{n}{\Cat}$\; of $\psc{n}{\Cat}$ as follows:

    For $n=1$, $H\in \segpsc{}{\Cat}$ if $H_0$ is discrete and the Segal maps are isomorphisms: that is, for all $k\geq 2$
    \begin{equation*}
        H_k\cong\pro{H_1}{H_0}{k}
    \end{equation*}
    Note that, since $p$ commutes with pullbacks over discrete objects, there is a functor
   \begin{equation*}
   \begin{split}
       & \p{2}:\segpsc{}{\Cat} \rw \Cat\;, \\
       & (\p{2}X)_{k}=p X_k\;.
   \end{split}
   \end{equation*}
   That is the following diagram commutes:
   \begin{equation*}
    \xymatrix{
    \segpsc{}{\Cat} \ar@{^{(}->}[rr]\ar_{\p{2}}[d] && \psc{}{\Cat} \ar^{\ovl{p}}[d]\\
    \Cat \ar[rr] && \funcat{}{\Set}
    }
   \end{equation*}
   When $n>1$, $\segpsc{n}{\Cat}$ is the full subcategory of $\psc{n}{\Cat}$ whose objects $H$ satisfy the following: \mk
   \begin{itemize}
     \item [a)] \emph{Discreteness condition}:  $H_{\uk(0,i)}$ is discrete for all $\uk\in\Dmenop$ and $1 \leq i \leq n$.\mk
     \item [b)] \emph{Segal condition}: All Segal maps are isomorphisms
      \begin{equation*}
      H_{\uk}\cong\pro{H_{\uk(1,i)}}{H_{\uk(0,i)}}{k_i}
      \end{equation*}
      for all $\uk\in\Dmenop$, $1 \leq i \leq n$ and $k_i\geq2$.\mk
     \item [c)] \emph{Truncation functor}: There is a functor
     \begin{equation*}
        \p{n+1}:\segpsc{n}{\Cat}\rw \catwg{n}
     \end{equation*}
   \end{itemize}
    making the following diagram commute:
    \begin{equation*}
        \xymatrix{
        \segpsc{n}{\Cat} \ar@{^(->}[r] \ar_{\p{n+1}}[d] & \psc{n}{\Cat} \ar^{\ovl{p}}[d]\\
        \catwg{n} \ar_{\Nb{n}}[r] & \funcat{n}{\Set}
        }
    \end{equation*}
\end{definition}
\bigskip

The main property of Segalic pseudo-functors is the theorem below stating that that the classical strictification of pseudo-functors, when applied to a Segalic pseudo-functors, yields weakly globular \nfol categories.

\begin{theorem}\cite[Theorem 4.5]{Pa2}\label{the-strict-funct}
    The strictification functor
    \begin{equation*}
        \St  : \psc{n-1}{\Cat}\rw \funcat{n-1}{\Cat}
    \end{equation*}
    restricts to a functor
    \begin{equation*}
        L_n: \segpsc{n-1}{\Cat}\rw J_n\catwg{n}
    \end{equation*}
    where $J_n\catwg{n}$ denotes the image of the fully faithful functor $J_n:\catwg{n}\rw\funcat{n-1}{\Cat}$\;.
    Further, for each $H\in \segpsc{n-1}{\Cat}$ and $\uk\in\dop{n-1}$, the map $(L_n H)_{\uk}\rw H_{\uk}$ is an equivalence of categories.
\end{theorem}
\subsection{From weakly globular Tamsamani $n$-categories to Segalic pseudo-functors}\label{subs-from-Tam-to-segalic-pseudo}
\begin{definition}\label{def-ind-sub-ltawg}
    Define inductively the subcategory $\lta{n}\subset\tawg{n}$. For $n=2$, $\lta{2}=\tawg{2}$. Suppose we defined $\lta{n-1}\subset \tawg{n}$. Let $\lta{n}$ be the full subcategory of $\tawg{n}$ whose objects $X$ are such that
    \begin{itemize}
      \item [i)] $X_k\in\lta{n-1}$ for all $k\geq 0$.\bk

      \item [ii)] The maps in $\funcat{n-2}{\Cat}$
    \begin{equation*}
        v_k: J_{n-1} X_k \rw J_{n-1}(\pro{X_1}{\di{n-1}\p{n-1}X_0}{k})
    \end{equation*}
    are levelwise equivalences of categories for all $k\geq 2$\bk

     \item [iii)] $\p{n}X\in\catwg{n-1}$.
    \end{itemize}
\end{definition}
\begin{notation}\label{not-ind-seg-map}
    Let $X\in\tawg{n}$, $\uk=(k_1,\ldots,k_{n-1})\in\dop{n-1}$, $1\leq i\leq n-1$. Then there is $X^i_{\uk}\in\funcat{}{\Cat}$ with
    \begin{equation*}
        (X^i_{\uk})_r=X_{\uk(r,i)}=X_{k_1...k_{i-1}r k_{i+1}...k_{n-1}}
    \end{equation*}
    so that $(X^i_{\uk})_{k_i}=X_{\uk}$. Since $X_{k_1,\ldots,k_{i-1}}\in\tawg{n-i+1}$, $X_{k_1,\ldots,k_{i-1}0}\in\cathd{n-i}$ and thus by \cite[Lemma 3.1]{Pa1}
    \begin{equation*}
        X_{k_1...k_{i-1} 0 k_{i+1}...k_{n-1}}= X_{\uk(0,i)}\in\cathd\;.
    \end{equation*}
    We therefore obtain induced Segal maps in $\Cat$ for all $k_i\geq 2$.
    \begin{equation}\label{eq1-not-ind-seg-map}
        \nu(\uk,i):X_{\uk}\rw \pro{X_{\uk(1,i)}}{X^d_{\uk(0,i)}}{k_i}\;.
    \end{equation}
\end{notation}
\begin{lemma}\cite[Lemma 7.2]{Pa3}\label{lem-maps-nu-eqcat}
    Let $X\in\lta{n}$; then for each $\uk\in\dop{n-1}$, $1\leq i\leq n-1$ and $k_i\geq 2$ the maps $ \nu(\uk,i)$ as in \eqref{eq1-not-ind-seg-map} are equivalences of categories.
\end{lemma}
\begin{theorem}\cite[Theorem 7.3]{Pa3}\label{the-XXXX}
    There is a functor
    \begin{equation*}
        Tr_{n}: \lta{n} \rw \segpsc{n-1}{\Cat}
    \end{equation*}
    together with a pseudo-natural transformation $t_n (X):Tr_{n}X\rw X$ for each $X\in\lta{n}$ which is a levelwise equivalence of categories.
\end{theorem}
As explained in \cite{Pa3}, for each $X\in\lta{n}$, $\uk\in\dop{n-1}$, $1\leq i\leq n-1$ we have
\begin{itemize}
  \item [i)] If $k_j=0$ for some $1\leq j\leq n-1$
  \begin{equation*}
    (Tr_{n} X)_{\uk}=X^d_{\uk}\;.
  \end{equation*}

  \item [ii)] If $k_j\neq 0$ for all $1\leq j\leq n-1$ and $k_i=1$,
  \begin{equation*}
    (Tr_{n} X)_{\uk}=X_{\uk(1,i)}=X_{\uk}\;.
  \end{equation*}

  \item [iii)] If $k_j\neq 0$ for all $1\leq j\leq n-1$ and $k_i>1$,
  \begin{equation*}
    (Tr_{n} X)_{\uk}=\pro{X_{\uk(1,i)}}{X^d_{\uk(0,i)}}{k_i} \;.
  \end{equation*}

  \end{itemize}
 \bigskip
We now prove an additional property of the category $\lta{n}$ which we will be needed in Section \ref{sec-fta-to-tam}.
\begin{lemma}\label{lem-property-lta}
    Let $X,Y\in\lta{n}$ be such that
    \begin{itemize}
      \item [i)] For each $\uk\in\dop{n-1}$ such that $k_j=0$ for some $1\leq j\leq n-1$, $X^d_{\uk}\cong Y^d_{\uk}$.

      \item [ii)] For each $\uk\in\dop{n-1}$ such that $k_j\neq 0$ for all $1\leq j\leq n-1$, $X^d_{\uk}\cong Y^d_{\uk}$.
    \end{itemize}
    Then $Tr_n X\cong Tr_n Y$
\end{lemma}
\begin{proof}
Let $\uk\in\dop{n-1}$ with $k_j=0$ for some $j$. Then by definition
\begin{equation*}
    (Tr_n X)_{\uk}= X_{\uk}^d\cong Y_{\uk}^d=(Tr_n Y)_{\uk}\;.
\end{equation*}
Let $\uk\in\dop{n-1}$ be such that $k_j\neq 0$ for all $j$ and suppose $k_i=1$. Then
\begin{equation*}
    (Tr_n X)_{\uk}= X_{\uk(1,i)}  = X_{\uk} \cong  Y_{\uk} = Y_{\uk(1,i)} =(Tr_n Y)_{\uk}\;.
\end{equation*}
If $k_j >1$,
\begin{align*}
    (Tr_n X)_{\uk} & = \pro{X_{\uk(1,i)}}{X^d_{\uk(0,i)}}{k_i}\cong\\
    & \cong \pro{Y_{\uk(1,i)}}{Y^d_{\uk(0,i)}}{k_i}=(Tr_n Y)_{\uk}\;.
\end{align*}
In conclusion
\begin{equation*}
    (Tr_n X)_{\uk}\cong(Tr_n Y)_{\uk}
\end{equation*}
for all $\uk\in\dop{n-1}$. Hence
\begin{equation*}
    Tr_n X \cong Tr_n Y\;.
\end{equation*}
\end{proof}
\subsection{Rigidifying weakly globular Tamsamani $n$-categories}\label{subs-rig-Tam-cat}
In paper \cite{Pa3} the author proved the existence of a rigidification functor
\begin{equation*}
  \Qn:\tawg{n}\rw \catwg{n}
\end{equation*}
approximating each $X\in\tawg{n}$ with an $n$-equivalent $\Qn X\in\catwg{n}$.
\begin{theorem}\cite[Theorem 7.4]{Pa3}\label{the-funct-Qn}
    There is a functor
    \begin{equation*}
        Q_n:\tawg{n} \rw \catwg{n}
    \end{equation*}
    and for each $X\in\tawg{n}$ a morphism in $\tawg{n}$  $s_n(X):Q_n X\rw X$, natural in $X$, such that $(s_n(X))_{k}$ is a $(n-1)$-equivalence for all $k\geq0$. In particular, $s_{n}(X)$ is an $n$-equivalence.
\end{theorem}
The construction of $\Qn$ is by induction on $n$. When $n=2$, $Q_2$ is the composite
\begin{equation*}
  Q_2:\tawg{2}\xrw{Tr_2} \segpsc{}{\Cat}\xrw{\St}\catwg{2}
\end{equation*}
Suppose, inductively, that we defined $\Qnm$. Define the functor
\begin{equation*}
    P_n:\tawg{n}\rw\lta{n}
\end{equation*}
as follows. Given $X\in\tawg{n}$, consider  the pullback in $\funcat{n-1}{\Cat}$
\begin{equation*}
    \xymatrix@C=50pt@R=30pt{
    P_n X \ar^{w(X)}[r] \ar[d] & X \ar^{\zg_{n}}[d] \\
    \di{n}\Qnm\q{n}X \ar_{s_{n-1}(\q{n}X)}[r] & \di{n}\q{n}X
    }
\end{equation*}
When $n>2$ we define $\Qn$ to be the composite
\begin{equation*}
  \Qn:\tawg{n}\xrw{P_n}\lta{n}\xrw{Tr_n}\segpsc{n-1}{\Cat}\xrw{\St}\catwg{n}
\end{equation*}
where $Tr_n$ is as in Theorem \ref{the-XXXX} and $\St$ lands in $\catwg{n}$ by Theorem \ref{the-strict-funct}.
\section{Canonical choices of homotopically discrete objects}\label{sec-canonical}
The goal of this section is to show how to approximate up to $n$-equivalence a weakly globular \nfol category with a better behaved one in which the homotopically discrete object at level 0 admits a canonical choice of section to the discretization map. This will be used in Section \ref{sec-fta} to construct the category $\ftawg{n}$ which will then lead in Section \ref{sec-fta-to-tam} to the discretization functor from $\catwg{n}$ to $\ta{n}$.
\subsection{A general construction}\label{sbs-gen-const}
Let $\clC$ be a category with finite limits; let $X\in\Cat\clC$ and $f_0:Y_0\rw X_0$ be a morphism in $\clC$. There is $X(f_0)\in\Cat\clC$ with $X(f_0)_0=Y_0$ and $X(f_0)_1$ given by the pullback in $\clC$
\begin{equation*}
    \xymatrix{
    X(f_0)_1 \ar^{}[rr] \ar^{}[d] && Y_0\times Y_0 \ar^{f_0\times f_0}[d]\\
    X_1 \ar^{}[rr] && X_0\times X_0
    }
\end{equation*}
Further, for each $k\geq 2$, there is a pullback in $\clC$
\begin{equation*}
    \xymatrix@R=40pt{
    X(f_0)_k=\pro{X(f_0)_1}{Y_0}{k} \ar^{}[rr] \ar^{}[d] && \pro{Y_0}{}{k+1} \ar^{\pro{f_0}{}{k+1}}[d]\\
    X_k=\pro{X_1}{X_0}{k} \ar^{}[rr] && \pro{X_0}{}{k+1}
    }
\end{equation*}
\begin{lemma}\label{lem-gen-const-1}
Let
\begin{equation*}
    \xymatrix{
    A \ar[r] \ar[d] & B \ar^{f}[d]\\
    C \ar_{g}[r] & D
    }
\end{equation*}
be a pullback in $\Cat$ with $f$ an isofibration. Then
\begin{equation*}
    \xymatrix{
    pA \ar[r] \ar[d] & pB \ar^{}[d]\\
    pC \ar_{}[r] & pD
    }
\end{equation*}
is a pullback in $Set$.
\end{lemma}

\begin{proof}
Since $f$ is an isofibration, by \cite{Js} $A$ is equivalent to the pseudo-pullback $\ds A \simeq C \tms{ps}{D} B$. The functor $p$ sends pseudo-pullbacks to pullbacks. In fact, suppose we are given a commuting diagram in $\Set$
\begin{equation*}
    \xymatrix{
    X \ar@/_/[ddr]_r \ar@/^/[drr]^s \\
    & p(\ds{C\tms{ps}{D} B}) \ar[d] \ar[r]  & pB \ar^{pf}[d] \\
    & pC \ar_{pg}[r] & pD }
\end{equation*}
$(pg)_r =(pf)_s$. If we choose maps $b:dpB\rw B$ and $c:dpC\rw C$ we then have $fdbs \cong gdcr$. Therefore, there is $\ds v:X \rw C \tms{ps}{D} B$ such that $p_1v=bs$ and $p_2 v=cr$. Hence
\begin{equation*}
    p(p_1)p(v)=p(b)p(s)=p(s),\qquad p(p_2)p(v)=p(c)p(r)=p(r)\;.
\end{equation*}
This shows that
\begin{equation*}
    p(\ds{C\tms{ps}{D} B})=pC\tiund{pD}pB\;.
\end{equation*}
\end{proof}
\begin{lemma}\label{lem-gen-const-2}
    Let
    \begin{equation*}
        \xymatrix{
         A \ar^{s}[r] \ar_{r}[d] & B \ar^{f}[d]\\
    C \ar_{g}[r] & D
        }
    \end{equation*}
    be a pullback in $\Cat$, and suppose that $f$ is fully faithful. Then so is $r$.
\end{lemma}
\begin{proof}
For all $x,y\in A_0$ there is a pullback in $\Cat$
\begin{equation*}
    \xymatrix{
    A(x,y) \ar[rr] \ar[d] && B(sx,sy) \ar[d]\\
    C(rx.ry) \ar[rr] && D(grx,gry)=D(fsx,fsy)
    }
\end{equation*}
Since the right vertical map is an isomorphism (as $f$ is fully faithful), so is the left vertical map, showing that $r$ is fully faithful.
\end{proof}
\begin{lemma}\label{lem-gen-const-3}
    Let
\begin{equation*}
    \xymatrix{
    P \ar[r] \ar[d] & C \ar^{f}[d]\\
    A \ar_{s}[r] & B
    }
\end{equation*}
be a pullback in $\Cat$ with $f$ an isofibration and with $A,B,C \in \cathd{}$. Then $P\in\cathd{}$.
\end{lemma}
\begin{proof}
Since $f$ is an isofibration and $A\simeq A^d$, $B\simeq B^d$, $C\simeq C^d$, we have
\begin{equation*}
    P\simeq A \tms{ps}{B}C \simeq A^d\tiund{B^d} C^d
\end{equation*}
and therefore $P\in\cathd{}$.
\end{proof}
\begin{lemma}\label{lem-gen-const-4}
        Let
\begin{equation*}
    \xymatrix{
    P \ar[r] \ar_{h}[d] & C \ar^{f}[d]\\
    A \ar_{g}[r] & B
    }
\end{equation*}
be a pullback in $\funcat{n-1}{\Cat}$ with $A,B,C \in \catwg{n}$ and $f$ an $n$-equivalence which is a levelwise isofibration in $\Cat$ and the same holds for $p\up{r,n}f$ for all $1\leq r\leq n$. Then $h$ is an $n$-equivalence.
\end{lemma}
\begin{proof}
By induction on $n$. When $n=1$, since $f$ is an isofibration, $P$ is equivalent to the pseudo-pullback $A\tms{ps}{B}C$; since $f$ is an equivalence of categories, the latter is equivalent to $A$.

Suppose, inductively, that the lemma holds for $(n-1)$. By hypothesis, there is a pullback in $\Cat$
\begin{equation*}
    \xymatrix{
    p\up{2,n}P \ar[rr] \ar[d] && p\up{2,n}C \ar[d]\\
    p\up{2,n}A \ar[rr] && p\up{2,n}B
    }
\end{equation*}
and therefore, at object level, a pullback in $\Set$
\begin{equation*}
    \xymatrix{
    P_0^d \ar[r] \ar_{}[d] & C_0^d \ar^{}[d]\\
    A_0^d \ar_{}[r] & B_0^d
    }
\end{equation*}
Let $(a,c),(a',c')\in P_0^d$. Then there is a pullback in $\catwg{n-1}$
\begin{equation*}
    \xymatrix{
    P((a,c),(a',c')) \ar[rr] \ar_{h((a,c),(a',c'))}[d] && C(c,c') \ar^{}[d]\\
    A(a,a') \ar_{}[rr] && B(ga,ga')
    }
\end{equation*}
By hypothesis, this satisfies the induction hypothesis and thus $h((a,c),(a',c'))$ is a $(n-1)$-equivalence.

We also have the pullback in $\Set$
\begin{equation*}
    \xymatrix{
    p\up{1,n}P \ar[rr] \ar_{p\up{1,n}h}[d] && p\up{1,n}C \ar^{p\up{1,n}f}[d]\\
    p\up{1,n}A \ar[rr] && p\up{1,n}B
    }
\end{equation*}
Since $f$ is a $n$-equivalence, $p\up{1,n}f$ is surjective, therefore such is $p\up{1,n}h$. By Proposition 4.11 of \cite{Pa3} we conclude that $h$ is an $n$-equivalence.
\end{proof}
\begin{proposition}\label{pro-gen-const-1}
    Let $X\in\catwg{n}$ and $f_0:Y_0\rw X_0$ be a morphism in $\cathd{n-1}$ such that, for each $1\leq r\leq n$, $J_{n-1}f_0$ and $J_r p\up{n,r}f_0$ is a levelwise isofibration in $\Cat$ which is surjective on objects. Then
    \begin{itemize}
      \item [a)] $X(f_0)\in\catwg{n}$;\mk

      \item [b)] $V(X):X(f_0)\rw X$ is an $n$-equivalence;\mk

      \item [c)] if $X\in\cathd{n}$, $X(f_0)\in\cathd{n}$.
    \end{itemize}
\end{proposition}
\begin{proof}
By induction on $n$. Let first $n=2$. Since $Y_0\in\cathd{}$, to show that $X(f_0)\in\catwg{2}$ it is enough to show (by Lemma 3.14 of \cite{Pa2}) that $\ovl{p}X(f_0)\in\Cat$. That is, for each $k\geq 2$
\begin{equation*}
    p(X(f_0))_k=\pro{p(X(f_0))_1}{p(X(f_0))_0}{k}\;.
\end{equation*}
We show this for $k=2$, the case $k>2$ being similar.
From the general construction \ref{sbs-gen-const},
\begin{equation*}
    X(f_0)_2=(\tens{X_1}{X_0})\tiund{X_0\times X_0\times X_0} (Y_0\times Y_0\times Y_0)\;.
\end{equation*}
Since $f_0$ is an isofibration, using Lemma \ref{lem-gen-const-1}, the fact that $p\up{2}X\in\Cat$ and the fact that $p$ preserves products, we obtain
\begin{align*}
    &pX(f_0)_2=p(\tens{X_1}{X_0})\tiund{p(X_0\times X_0\times X_0)} p(Y_0\times Y_0\times Y_0)=\\
    =&(\tens{pX_1}{pX_0})\tiund{pX_0\times pX_0\times pX_0} pY_0\times pY_0\times pY_0\;.
\end{align*}
On the other hand,
\begin{align*}
    & \ \tens{pX(f_0)_1}{pX(f_0)_0}=\\
    = & \ \tens{(pX_1\tiund{\tens{pX_0}{}}(\tens{pY_0}{}))}{pX_0\tiund{pX_0}pY_0}=\\
    = & \ (\tens{pX_1}{pX_0})\tiund{\tens{(\tens{pX_0}{}){}} {pY_0}}\tens{(\tens{pY_0}{}){}}{pY_0}=\\
    = & \ p(\tens{X_1}{X_0})\tiund{pX_0\times pX_0\times pX_0}(pY_0\times pY_0\times pY_0)\;.
\end{align*}
Therefore
\begin{equation*}
    pX(f_0)_2=\tens{pX(f_0)_1}{pX(f_0)_0}\;.
\end{equation*}
The case $k>2$ is similar. This shows $X(f_0)\in\catwg{2}$, proving a) when $n=2$.

We now show that $X(f_0)\rw X$ is a 2-equivalence. Let $a,b \in Y_0^d$. We have a pullback in $\Cat$
\begin{equation}\label{eq1-gen-const}
    \xymatrix{
    X(f_0)(a,b) \ar[rr] \ar_{V(X)(a,b)}[d] && Y_0(a)\times Y_0(b) \ar[d]\\
    X_1(f_0 a,f_0 b) \ar[rr] && X_0(f_0 a)\times X_0(f_0 b)
    }
\end{equation}
Since $X_0,Y_0\in\cathd{}$, $Y_0(a)\rw X_0(a)$ is an equivalence of categories hence it is in particular fully faithful. Applying Lemma \ref{lem-gen-const-1} to \eqref{eq1-gen-const} we obtain a pullback in $\Set$
\begin{equation}\label{eq2-gen-const}
    \xymatrix{
   pX(f_0)(a,b) \ar[rr] \ar_{}[d] && pY_0(a)\times pY_0(b) \ar[d]\\
    pX_1(f_0 a,f_0 b) \ar[rr] && pX_0(f_0 a)\times pX_0(f_0 b)
    }
\end{equation}
Since, by hypothesis, $Y_0\rw X_0$ is surjective on objects, the right vertical map in \eqref{eq2-gen-const} is surjective, therefore such is the left vertical map in \eqref{eq2-gen-const}.
Thus $X(f_0)(a,b) \rw X_1(f_0 a,f_0 b)$ is essentially surjective on objects and in conclusion it is an equivalence of categories.

To show that $V(X):X(f_0)\rw X$ is a 2-equivalence, it is enough to show (by Proposition 4.11 of \cite{Pa3}) that $pp\up{2}V(X)$ is surjective. This follows from the fact that $pf_0$ is surjective (since by hypothesis $f_0$ is surjective on objects), so that $p\up{2}V(X)$ is surjective on objects. This concludes the proof of b) in the case $n=2$.

Finally, if $X\in\cathd{2}$, since by a)  $X(f_0)\in\catwg{2}$ and $V(X)$ is a 2-equivalence, it follows from Proposition 3.10 of \cite{Pa2} that $X(f_0)\in\cathd{2}$, proving c) in the case $n=2$.

Suppose, inductively, that the proposition holds for $(n-1)$, let $X\in\catwg{n}$ and $f_0$ be as in the hypothesis.\bk

a) We show that $X(f_0)\in\catwg{n}$ by proving that it satisfies the hypothesis of Proposition 3.16 of  \cite{Pa2}. By the general construction \ref{sbs-gen-const}, $X(f_0)\in\cat{n}$ and $(X(f_0))_0=Y_0\in\cathd{n-1}$. Since $X\in\catwg{n}$, $X_{\bl 0}\in\cathd{n-1}$ and we have a pullback in $\funcat{n-2}{\Cat}$
\begin{equation}\label{eq3-gen-const}
    \xymatrix{
    X(f_0)_{10} \ar[rr] \ar[d] && Y_{00}\times Y_{00}\ar^{f_{00}\times f_{00}}[d]\\
    X_{10} \ar[rr] && X_{00}\times X_{00}
    }
\end{equation}
where, by hypothesis, $J_{n-1}f_{00}$ and $J_r p\up{n-2,r} f_{00}$ is levelwise an isofibration in $\Cat$, which is surjective on objects. Thus \eqref{eq3-gen-const} satisfies the induction hypothesis c) and we conclude that $(X(f_0))_{\bl 0}\in\cathd{n-1}$. In particular, $(X(f_0))_{s0}\in\cathd{n-2}$. It remains to show that $\ovl{p}J_n X(f_0)\in H\lo{n-1}\catwg{n-1}$. Let $\ur\in \dop{n-2}$; by the construction \ref{sbs-gen-const},
\begin{equation*}
    (X(f_0))_{2\ur}=  \{\tens{X_{1\ur}}{X_{0\ur}}\}\tiund{X_{0\ur}\times X_{0\ur}\times X_{0\ur}} \{Y_{0\ur}\times Y_{0\ur}\times Y_{0\ur}\}\;.
\end{equation*}
Since, by hypothesis, $(f_0)_{\ur}$ is an isofibration, by Lemma \ref{lem-gen-const-1} we obtain
\begin{equation*}
    p(X(f_0))_{2\ur}= p \{\tens{X_{1\ur}}{X_{0\ur}}\}\tiund{p \{X_{0\ur}\times X_{0\ur}\times X_{0\ur}\}} p\{Y_{0\ur}\times Y_{0\ur}\times Y_{0\ur}\}\;.
\end{equation*}
Since $X\in \catwg{n}$, $\p{n}X \in \catwg{n-1}$ and since $(\p{n-1}X_j)_{\ur}=pX_{j\ur}$ we have
\begin{equation*}
    p \{\tens{X_{1\ur}}{X_{0\ur}}\}=\tens{p(X_{1\ur})}{p(X_{0\ur})}
\end{equation*}
As $p$ commutes with products, we obtain
\begin{equation*}
    p(X(f_0))_{2\ur}=\{\tens{pX_{1\ur}}{pX_{0\ur}}\}\tiund{pX_{0\ur}\times pX_{0\ur}\times pX_{0\ur}}\{pY_{0\ur}\times pY_{0\ur}\times pY_{0\ur}\}\;.
\end{equation*}
Since this holds for all $\ur$, it follows that
\begin{equation*}
    \ovl{p}X(f_0)_2=\tens{\ovl{p}X(f_0)_1}{\ovl{p}X(f_0)_0}\;.
\end{equation*}
Similarly for each $k>2$,
\begin{equation*}
    \ovl{p}X(f_0)_k=\pro{\ovl{p}X(f_0)_1}{\ovl{p}X(f_0)_0}{k}\;.
\end{equation*}
We conclude that
\begin{equation*}
    \ovl{p}J_n X(f_0)=(\p{n}X)(\p{n-1}f_0)\;.
\end{equation*}
The pullback
\begin{equation*}
    \xymatrix{
    (\p{n}X)(\p{n-1}f_0) \ar[rr]\ar[d] && \p{n-1}Y_0\times \p{n-1}Y_0 \ar[d]\\
    \p{n-1}X_1 \ar[rr] && \p{n-1}X_0\times \p{n-1}X_0
    }
\end{equation*}
satisfies the inductive hypothesis a) and we therefore conclude that $\ovl{p}J_n X(f_0)\in\catwg{n}$. Thus $X(f_0)$ satisfies the hypotheses of Proposition 3.16 of \cite{Pa2} and we conclude that $X(f_0)\in\catwg{n}$ proving a).\bk

b) Let $a,b\in Y_0^d$. There is a pullback in $\catwg{n-1}$
\begin{equation*}
    \xymatrix{
    X(f_0)(a,b) \ar[rr]\ar[d] && Y_0(a)\times Y_0(b) \ar[d]\\
    X(f_0 a,f_0 b)\ar[rr] && X_0(f_0 a)\times X_0(f_0 b)
    }
\end{equation*}
Since $X_0(f_0 a),\, Y_0(a)\in\cathd{n-1}$ and $Y_0(a)^d=\{a\}\cong \{f_0 a\}=X_0(f_0 a)^d$, the map $Y_0(a)\rw X_0(f_0 a)$ is an $n$-equivalence (by Lemma 3.8 of \cite{Pa1}). Also, this map is a levelwise isofibration in $\Cat$ (as such is $Y_0\rw X_0$). It follows from Lemma \ref{lem-gen-const-4} that $X(f_0)(a,b)\rw X(f_0 a,f_0 b)$ is a $(n-1)$-equivalence.

Finally, since $f_0$ is levelwise surjective on objects by hypothesis, $\p{1,n}X(f_0)\rw \p{1,n}X$ is surjective. By Proposition 4.11 of \cite{Pa3}, the map $X(f_0)\rw X$ is therefore an $n$-equivalence.\bk

c) This follows from a) and b) using Proposition 3.16 of \cite{Pa2}.
\end{proof}
\begin{lemma}\label{lem-gen-const-6}
    If $X\in\cathd{}$, $\Dec X\in\cathd{}$ and the map $d_1:\Dec X\rw X$ is an isofibration.
\end{lemma}
\begin{proof}
Since $X\in\cathd{}$, $X=A[f]$ for a surjective map of sets $f:A\rw B$, where $A[f]$ is as in Definition 4.1 of \cite{Pa1} . Thus $\Dec X=(\tens{A}{B})[d_0]$ where $d_0:\tens{A}{B}\rw A$, $d_0(x,y)=X$. The source and target maps
\begin{equation*}
    {\tld}_0, {\tld}_1:(\Dec X)_1=A\tiund{B}A\tiund{B}A\rw(Dec X)_0=\tens{A}{B}
\end{equation*}
are ${\tld}_0(x,y,z)=(x,y)$, ${\tld}_1(x,y,z)=(x,z)$.

Given $(x,y)\in(\Dec X)_0$ and an isomorphism $(d_1(x,y)=y,z)\in X_1$, we have $(x,y,z)\i(\Dec X)_1$ with $\tld_1(x,y,z)=(x,z)$. In picture:
\begin{equation*}
    \xymatrix{
    (x,y)=\tld_0(x,y,z) \ar^{(x,y,z)}[rr] \ar[d] && (x,z)=\tld_1(x,y,z)\ar[d]\\
    d_1(x,y)=y \ar_{(y,z)=d_1(x,y,z)}[rr] && d_1(x,z)=z
    }
\end{equation*}
By definition, this shows that $d_1:\Dec X\rw X$ is an isofibration. It is also surjective on objects since $d_1:\tens{A}{B}\rw A$ is surjective.
\end{proof}
\newpage
\begin{lemma}\label{lem-gen-const-6}\
    \begin{itemize}
      \item [a)] Let $A\in\cathd{n}$, $B,C\in\Set$ and consider the pullback in $\funcat{n-1}{\Cat}$
          \begin{equation*}
            \xymatrix{
            Q \ar[r] \ar[d] & A\ar[d]\\
            \di{n,1}C \ar[r] & \di{n,1}B
            }
          \end{equation*}
          then $Q\in\cathd{n}$.
          \bk

      \item [b)] Let $X\in\cathd{n}$, $Z\in\cathd{}$ and consider the pullback in $\funcat{n-1}{\Cat}$
          \begin{equation*}
            \xymatrix{
            P \ar[r] \ar[d] & X\ar[d]\\
            \di{n,2}Z \ar[r] & \di{n,2}\q{2,n}X
            }
          \end{equation*}
          Then $P\in\cathd{n}$.
    \end{itemize}
\end{lemma}
\begin{proof}
\nid By induction on $n$.

\nid In the case $n=1$ for a). Since $dB$ is discrete, $A\rw dB$ is an isofibration, therefore
\begin{equation*}
   Q=A\tiund{dB}dC \simeq    A\tms{ps}{dB}dC \simeq A^d\tiund{dB}dC\;.
\end{equation*}
Hence $Q\in\cathd{}$. The case $n=2$ for b) is Lemma 5.3 of \cite{Pa3}.

\bk
Suppose, inductively, that the lemma holds for $(n-1)$.

\nid a) for each $k\geq 0$ there is a pullback in $\funcat{n-2}{\Cat}$
\begin{equation*}
    \xymatrix{
    Q_k \ar[r] \ar[d] & A_k\ar[d]\\
            \di{n-1,1}C_k \ar[r] & \di{n-1,1}B_k
    }
\end{equation*}
Therefore, by inductive hypothesis a), $Q_k\in\cathd{n-1}$. For each $\ur=(r_1,...,r_{n-1})\in\dop{n-1}$, we have a pullback in $\Cat$
\begin{equation*}
    \xymatrix{
    Q_{\ur} \ar[r] \ar[d] & A_{\ur}\ar[d]\\
            d C \ar[r] & d B
    }
\end{equation*}
Since $p$ commutes with fiber products over discrete objects, we have a pullback in $\Set$
\begin{equation*}
    \xymatrix{
    p Q_{\ur} \ar[r] \ar[d] & p A_{\ur}\ar[d]\\
             C \ar[r] &  B
    }
\end{equation*}
It follows that there is a pullback in $\funcat{n-2}{\Cat}$
\begin{equation*}
    \xymatrix{
    \p{n} Q \ar[r] \ar[d] & \p{n} A\ar[d]\\
            \di{n-1,1} C \ar[r] & \di{n-1,1} B
    }
\end{equation*}
It follows by inductive hypothesis a) that $\p{n}Q\in\cathd{n}$. By definition, this means that $Q\in\cathd{n}$.
\bk

\nid b) For each $k\geq 0$, there is a pullback in $\funcat{n-2}{\Cat}$
\begin{equation*}
    \xymatrix{
    P_k \ar[r] \ar[d] & X_k \ar[d]\\
            \di{n-1,1} Z_k \ar[r] & \di{n-1,1}\q{1,n-1}X_k B
    }
\end{equation*}
Since $X_k\in\cathd{n-1}$ (as $X\in\cathd{n}$), by part a) this implies that $P_k\in\cathd{n-1}$. Since $\p{n}$ commutes with fiber products over discrete objects, we also have a pullback in $\funcat{n-2}{\cat}$
\begin{equation*}
    \xymatrix{
    \p{n} P \ar[r] \ar[d] & \p{n} X \ar[d]\\
            \di{n-1,2}\q{2,n} Z \ar[r] & \di{n-1,2}\q{2,n}X B
    }
\end{equation*}
where $\p{n}X\in\cathd{n-1}$. By inductive hypothesis b), $\p{n}P\in\cathd{n-1}$. Hence by definition, $P\in\cathd{n}$.
\end{proof}
\begin{proposition}\label{pro-gen-const-2}
    There is a functor $V_n:\cathd{n}\rw\cathd{n}$ with a map $f_X:V_n X\rw X$ natural in $X\in\cathd{n}$ such that
    \begin{itemize}
      \item [a)] $J_n f_X$ is a levelwise isofibration in $X\in\cathd{n}$ which is surjective on objects, and the same holds for $J_r \p{n,r} f_X$ for all $1\leq r\leq n$.\bk

      \item [b)] $V_n$ is identity on discrete objects and preserves pullbacks over discrete objects.\bk

      \item [c)] If $h:X\rw Y$ is a morphism in $\cathd{n}$, the following diagram  commutes for appropriate choices of sections to the discretization maps.
          \begin{equation*}
            \xymatrix{
            V_n X \ar[rr] && V_n Y\\
            (V_n X)^d \ar[rr]\ar[u] && (V_n Y)^d\ar[u]
            }
          \end{equation*}
    \end{itemize}
\end{proposition}
\begin{proof}
By induction on $n$. For $n=1$, let $V_1 X=\Dec X$ and $f_X=d_1:\Dec X\rw X$. By Lemma \ref{lem-gen-const-6}, $f_X$ is an isofibration and is surjective on objects. Also $\Dec$ preserves pullbacks. Given a morphism $h:X\rw Y$ in $\cathd{}$, we have a diagram
\begin{equation*}
    \xymatrix{
    \Dec X \ar[rr]\ar[d] && \Dec Y \ar[d]\\
    (\Dec X)^d=d X_0 \ar[rr] && (\Dec Y)^d=d Y_0
    }
\end{equation*}
This proves the lemma in the case $n=1$. Suppose, inductively, that it holds for $(n-1)$ and let $X\in\cathd{n}$.\bk

\nid a) Let $F_n X=X(f_{X_0})$ where $f_{X_0}:V_{n-1}X_0 \rw X_0$. By inductive hypothesis a), $f_{X_0}$ satisfies the hypothesis of Proposition \ref{pro-gen-const-1} and thus $F_n X\in\cathd{n}$.

Consider the pullback in $\funcat{n-1}{\Cat}$
\begin{equation*}
    \xymatrix{
    V_n X \ar^{h}[rr]\ar[d] && F_n X\ar[d]\\
    \di{n,2}\Dec\q{2,n}F_n X\ar[rr] && \di{n,2}\q{2,n}F_n X
    }
\end{equation*}
Since $F_n X\in\cathd{n}$, $\q{2,n}F_n X\in\cathd{}$, hence $\Dec \q{2,n}F_n X\in\cathd{}$. Thus by Lemma \ref{lem-gen-const-6}, $V_n X\in\cathd{n}$.

For each $\uk\in\dop{n-1}$, there is a pullback in $\Cat$
\begin{equation}\label{eq4-gen-const}
    \xymatrix{
    (V_n X)_{\uk} \ar^{}[rr]\ar[d] && (F_n X)_{\uk}\ar[d]\\
    d(\di{n,2}\Dec\q{2,n}F_n X)_{\uk} \ar[rr] && d(\di{n,2}\q{2,n}F_n X)_{\uk}\;.
    }
\end{equation}
The bottom horizontal map is an isofibration since the target is discrete; hence $H_{\uk}$ is also an isofibration. The bottom horizontal map in \eqref{eq4-gen-const} is also surjective on objects since
\begin{equation*}
    (N\Dec \q{2,n}F_n X)_r \rw (H\q{2,n}F_n X)_r
\end{equation*}
is surjective for all $r\geq 0$, where $N:\Cat\rw\funcat{}{\Set}$ is the nerve functor. It follows that $h_{\uk}$ is also surjective on objects. Since, by Proposition \ref{pro-gen-const-1}, the map $v_X:F_X \rw X$ is a levelwise isofibration surjective on objects, we conclude from above that the same holds for the composite map
\begin{equation*}
    f_X:V_n X \xrw{\;\;h\;\;}F_n X\xrw{v_X} X\;.
\end{equation*}
We now show that $\p{n,r}f_X$ is an isofibration surjective on objects for all $r$. Since, by Proposition \ref{pro-gen-const-1}, this holds for $\p{n,r}V_X$, it is sufficient to show this for $\p{n,r}h$.

Since $\p{n.r}$ commutes with pullbacks over discrete objects, we have a pullback in $\funcat{r-1}{\Cat}$
\begin{equation*}
    \xymatrix{
    \p{n-r}V_n X \ar^{\p{n-r}h}[rr]\ar[d] && \p{n-r}F_n X\ar[d]\\
    \di{r,2}\Dec\q{2,n}F_n X\ar[rr] && \di{r,2}\q{2,n}F_n X\;.
    }
\end{equation*}
Using a similar argument as above we conclude that $\p{n,r}h$ is a levelwise isofibration surjective on objects. This proves a).\bk

\nid b) If $X$ is discrete, $F_n X=X=\di{n,2}\q{2,n}X$, thus $V_n X=X$. Since, by inductive hypothesis, $V_{n-1}$ commutes with pullbacks over discrete objets, so does $F_n$ as easily seen. Since $\q{2,n}$ commutes with pullbacks over discrete objects and $\Dec$ commutes with pullbacks, it follows by construction that $V_n$ commutes with pullbacks over discrete objects.\bk

\nid c) We have
\begin{align*}
    (V_n X)^d & = (\q{2,n}\di{n,2}\Dec \q{2,n}F_n X)^d =\\
    & = (\Dec \q{2,n}F_n X)^d = (\q{2,n}F_n X)_0=\q{1,n-1}(F_n X)_0 =\\
    & = (F_n X)^d_0=(V_{n-1}X_0)^d
\end{align*}
and similarly for $V_n Y$. By inductive hypothesis c), there is a commuting diagram
\begin{equation*}
    \xymatrix{
    (V_n X)^d=(V_{n-1}X_0)^d \ar[rr]\ar[d] && V_{n-1}X_0=(F_n X)_0 \ar[rr]\ar[d] && F_n X \ar[d]\\
    (V_n Y)^d=(V_{n-1}Y_0)^d \ar[rr] && V_{n-1}Y_0=(F_n Y)_0 \ar[rr] && F_n Y
    }
\end{equation*}
as well as
\begin{equation*}
    \xymatrix{
    (V_n X)^d=(\Dec \q{2,n}F_n X)^d \ar[rr]\ar[d] && \di{n,2}\Dec q{2,n}F_n X\ar[d]\\
    (V_n Y)^d=(\Dec \q{2,n}F_n Y)^d \ar[rr] && \di{n,2}\Dec q{2,n}F_n Y\;.
    }
\end{equation*}
From the construction of $V_n X$ and $V_n Y$ we therefore conclude that there is a commuting diagram
\begin{equation*}
    \xymatrix{
    V_n X \ar[r] & V_n Y\\
    (V_n X)^d \ar[r]\ar[u] & (V_n Y)^d\ar[u]\;.
    }
\end{equation*}
\end{proof}
\begin{corollary}\label{cor-gen-const-1}
    There is a functor
    \begin{equation*}
        F_n:\catwg{n}\rw \catwg{n}
    \end{equation*}
    and a map $V_X:F_n X\rw X$ (natural in $X\in\catwg{n}$) such that
    \begin{itemize}
      \item [i)] $V_X$ is an $n$-equivalence.\bk

      \item [ii)] $F_n$ is identity on discrete objects and preserves pullbacks over discrete objects.\bk

      \item [iii)] If $f:X\rw Y$ is a morphism in $\catwg{n}$ the following diagram commutes for appropriate choices of sections to the discretization maps
          \begin{equation*}
            \xymatrix{
            (F_n X)_0 \ar[rr] && (F_n Y)_0\\
            (F_n X )^d_0 \ar[rr]\ar[u]  && (F_n Y)^d_0\ar[u]\;.
            }
          \end{equation*}
    \end{itemize}
\end{corollary}
\begin{proof}
Given $X\in\catwg{n}$, since $X_0\in\cathd{n-1}$ by Proposition \ref{pro-gen-const-2} there  is a map
\begin{equation*}
    f_{X_0}:V_{n-1}X_0 \rw X_0
\end{equation*}
such that $J_{n-1}f_{X_0}$ and $J_r p^{n-1,}f_{X_0}$ are levelwise isofibrations in $Cat$ surjective on objects. Let $F_n X=X(f_{X_0})$. By Proposition \ref{pro-gen-const-1}, $F_nX\in\catwg{n}$ and there is an $n$-equivalence $V(X):F_n X\rw X$, proving i). If $X$ is discrete, so is $X_0$, thus by Proposition \ref{pro-gen-const-2} $f_{X_0}=\Id$ and therefore $F_n X=X$.

Let $X\rw Z \lw Y$ be a pullback in $\catwg{n}$ with $Z$ discrete. By Proposition \ref{pro-gen-const-2}, $V_{n-1}(X_0\tiund{Z}Y_0)= V_{n-1}X_0\tiund{Z}V_{n-1}Y_0$ and therefore, as easily checked,
\begin{equation*}
    (F_n(X\tiund{Z}Y))_1=(F_n X)_1\tiund{Z}(F_n Y)_1\;.
\end{equation*}
It follows that
\begin{equation*}
    F_n(X\tiund{Z}Y)=F_n X\tiund{Z}F_n Y
\end{equation*}
which is ii). Since $(F_n X)_0=V_{n-1}X_0$, iii) follows from Proposition \ref{pro-gen-const-2}.
\end{proof}
\section{The category $\ftawg{n}$}\label{sec-fta}

The idea of the construction of the discretization functor is to replace the homotopically discrete substructures in $X\in\catwg{n}$ by their discretization. However, as outlined in the introduction to this paper, this cannot be done in a functorial way since the choice of sections to the discretization maps are not canonical.

For this reason, we introduce in this section a new category $\ftawg{n}$ in which the sections to the discretization maps for the homotopically discrete substructures are canonical.
We then show in Proposition \ref{pro-fta-1} that there is a functor
\begin{equation*}
        G_n:\catwg{n}\rw\ftawg{n}
    \end{equation*}
    and an $n$-equivalence in $\tawg{n}$, $G_nX \rw X$ for each $X\in \catwg{n}$. The construction of $G_n$ uses the functor $F_n$ built in Section \ref{sec-canonical}. In Section \ref{sec-fta-to-tam}, $G_n$ will be used to build the discretization functor $\Discn:\catwg{n}\rw \ta{n}$.

\begin{definition}\label{def-fta-1}
    Define the category $\ftawg{n}$ as follows. $\ftawg{1}=\Cat$. For each $n\geq 2$ let $\ftawg{n}$ have the following objects and morphisms:
    \begin{itemize}

      \item [i)]
    Objects of $\ftawg{n}$ consist of $X\in\tawg{n}$ such that for all $\uk=(\seqc{k}{1}{s)}\in\dop{s}$, $\ur=(\seqc{r}{1}{s})\in\dop{s}$, $(1\leq s \leq n-2)$ and morphism $\uk\rw\ur$ in $\dop{s}$, the corresponding morphism
    \begin{equation*}
        f:X_{\uk 0}\rw X_{\ur 0}
    \end{equation*}
    in $\cathd{n-s-1}$ is such that there are choices of sections to the discretization maps
    \begin{align*}
        &\zg(X_{\uk}):X_{\uk 0} \rw X_{\uk 0}^d\\
        &\zg(X_{\ur}):X_{\ur 0} \rw X_{\ur 0}^d\\
    \end{align*}
    making the following diagram commute
    \begin{equation}\label{eq-def-fta-1}
    \xymatrix@C=50pt@R=30pt{
    X_{\uk 0} \ar^{f}[r] & X_{\ur 0}\\
    X^d_{\uk 0} \ar_{f^d}[r] \ar^{\zg'(X_{\uk 0})}[u] & X^d_{\ur 0} \ar_{\zg'(X_{\ur 0})}[u]
    }
    \end{equation}

      \item [ii)] A morphism $F:X\rw Y$ in $\ftawg{n}$ is a morphism in $\tawg{n}$ such that, for all $\uk=(\seqc{k}{1}{s})\in\dop{s}$, $1\leq s \leq n-2$, the following diagram commutes
          \begin{equation}\label{eq-def-fta-2}
          \xymatrix@C=50pt@R=30pt{
          X_{\uk 0} \ar^{F_{\uk 0}}[r] & Y_{\uk 0}\\
          X^d_{\uk 0} \ar_{F^d_{\uk 0}}[r] \ar^{\zg'(X_{\uk 0})}[u] & Y^d_{\uk 0} \ar_{\zg'(Y_{\uk 0})}[u]
          }
          \end{equation}
    \end{itemize}
\end{definition}
\begin{remark}\label{rem-fta-1}
    It is immediate from the definition that, if $F\in\ftawg{n}$, $X_k\in\ftawg{n-1}$ for all $k\geq 0$.
\end{remark}
\begin{lemma}\label{lem-fta-1}
    The functors $\p{n},\q{n}:\tawg{n}\rw \tawg{n-1}$ induce functors
    \begin{equation*}
        \p{n},\q{n}:\ftawg{n}\rw \ftawg{n-1}\;.
    \end{equation*}
\end{lemma}
\begin{proof}
Let $X\in\ftawg{n}$ and $\uk\rw\ur$ be a morphism in $\dop{s}$. By applying the functor $\p{n-s-1}$ to the commuting diagram \eqref{eq-def-fta-1} and using the fact that
\begin{equation*}
    \p{n-s-1}X_{\uk 0}=(\p{n}X)_{\uk 0},\qquad \p{n-s-1}X^d_{\uk 0}= X^d_{\uk 0}=(\p{n}X)^d_{\uk 0}
\end{equation*}
we obtain the commuting diagram
\begin{equation*}
\xymatrix@C=60pt@R=40pt{
(\p{n}X)_{\uk 0} \ar^{\p{n-s-1}f}[r] & (\p{n}X)_{\ur 0} \\
(\p{n}X)^d_{\uk 0} \ar_{\p{n-s-1}f^d}[r] \ar[u] & (\p{n}X)^d_{\ur 0} \ar[u]
}
\end{equation*}
This shows that $\p{n}X\in\ftawg{n}$. Given $F:X\rw Y$ in $\ftawg{n}$, by applying $\p{n-s-1}$ to the commuting diagram \eqref{eq-def-fta-2} we obtain the commuting diagram
\begin{equation*}
\xymatrix@C=60pt@R=40pt{
(\p{n}X)_{\uk 0} \ar^{(\p{n}F)_{\uk 0}}[r] & (\p{n}Y)_{\uk 0} \\
(\p{n}X)^d_{\uk 0} \ar_{(\p{n}F)^d_{\uk 0}}[r] \ar[u] & (\p{n}Y)^d_{\uk 0} \ar[u]
}
\end{equation*}
By definition this means that $\p{n}F\in\ftawg{n}$. The proof for $\q{n}$ is analogous since $ \p{n-s-1}X_{\uk 0}= \q{n-s-1}X_{\uk 0}$.
\end{proof}
\begin{proposition}\label{pro-fta-1}
    For each $n\geq 2$ there is a functor
    \begin{equation*}
        G_n:\catwg{n}\rw\ftawg{n}
    \end{equation*}
    defined as follows. $G_2=F_2$; given $G_{n-1}$ let $G_n=\ovl{G}_{n-1}\circ F_n$. Then
    \begin{itemize}
      \item [a)] $G_n:\catwg{n}\rw \ftawg{n}$ and $(G_n X)_k\in\catwg{n-1}$.\mk

      \item [b)] There is an $n$-equivalence in $\tawg{n}$, $G_n X\rw X$, natural in $X\in\catwg{n}$.\mk

      \item [c)] $G_n$ preserves $n$-equivalences.\mk

      \item [d)] $G_n$ is identity on discrete objects and preserves pullbacks over discrete objects.

    \end{itemize}
\end{proposition}
\begin{proof}
For $n=2$, by Corollary \ref{cor-gen-const-1} the functor $F_2:\catwg{2}\rw \catwg{2}$ is in fact a functor $F_2:\catwg{2}\rw \ftawg{2}$ satisfying a) - d).

Suppose we defined $G_{n-1}$ satisfying the above properties and let $X\in\catwg{n}$. We first check that $G_n X\in\tawg{n}$. We have
\begin{equation*}
    (G_n X)_0=G_{n-1}(F_n X)_0\;.
\end{equation*}
Since $(F_n X)_0\in\cathd{n-1}$ (as $F_n X\in\catwg{n}$), there is a $(n-1)$-equivalence $(F_n X)\rw(F_n X)^d_0$. Thus by inductive hypothesis c) and d) this induces an $(n-1)$-equivalence
\begin{equation*}
    (G_n X)_0=G_{n-1}(F_n X)_0 \rw G_{n-1}(F_n X)^d_0 = (F_n X)^d_0\;.
\end{equation*}
Since, by inductive hypothesis a), $(G_n X)_0\in\catwg{n-1}$, it follows that
\begin{equation*}
    (G_n X)^d_0=(F_n X)^d_0\;.
\end{equation*}
For each $k>0$ by inductive hypothesis we also have
\begin{equation*}
    (G_n X)_k = G_{n-1}(F_n X)_k\in\catwg{n-1}\;.
\end{equation*}
To show that $G_n X\in\tawg{n}$ it remains to prove that the induced Segal maps are $(n-1)$-equivalences. Since $F_n X\in\catwg{n}$ there are $(n-1)$-equivalences
\begin{equation*}
    (F_n X)_2 \rw \tens{(F_n X)_1}{(F_n X)^d_0}\;.
\end{equation*}
Using the induction hypotheses c) and d) this induces an $(n-1)$-equivalence
\begin{equation*}
\begin{split}
  (G_n X)_2 & =G_{n-1}(F_n X)_2 \rw G_{n-1} \{\tens{(F_n X)_1}{(F_n X)^d_0}\}\cong \\
    & \cong \tens{(G_n X)_1}{(G_n X)^d_0}\;.
\end{split}
\end{equation*}
Similarly one shows that all other induced Segal maps for $G_n X$ are $(n-1)$-equivalences. We conclude that $G_n X\in\tawg{n}$.
\bk

a) Let $\uk=(\seqc{k}{1}{s})$, $\ur=(\seqc{r}{1}{s})$ in $\dop{s}$ and denote $\uk'=(\seqc{k}{2}{s})$, $\ur'=(\seqc{r}{2}{s})$ and suppose we have a morphism $\uk\rw\ur$ in $\dop{s}$. By factoring this as
\begin{equation*}
    \xymatrix@R=25pt{
    \uk=(k_1,\uk') \ar[rr] \ar[dr] && \ur=(r_1,\ur')\\
    & (r_1,\uk')\ar[ur]
    }
\end{equation*}
we obtain a factorization
\scriptsize
\begin{equation}\label{eq-pro-fta-1}
\xymatrix@C=5pt{
(G_n X)_{\uk 0}=\{G_{n-1}(F_n X)_{k_1}\}_{\uk' 0} \ar[rr]\ar[dr] && \{G_{n-1}(F_n X)_{r_1}\}_{\ur' 0}=(G_n X)_{\ur 0}\\
& \{G_{n-1}(F_n X)_{r_1}\}_{\uk' 0}\ar[ur]
}
\end{equation}
\normalsize
Consider the morphism $(F_n X)_{k_1}\rw (F_n X)_{r_1}$ in $\catwg{n-1}$. Since, by induction hypothesis a), $G_{n-1}:\catwg{n-1}\rw \ftawg{n-1}$ there is a commuting diagram
\begin{equation}\label{eq-pro-fta-2}
\xymatrix@C=50pt@R=40pt{
\{G_{n-1}(F_n X)_{k_1}\}_{\uk' 0} \ar[r] & \{G_{n-1}(F_n X)_{r_1}\}_{\uk' 0}\\
\{G_{n-1}(F_n X)_{k_1}\}^d_{\uk' 0} \ar[r]\ar[u] & \{G_{n-1}(F_n X)_{r_1}\}^d_{\uk' 0}\ar[u]
}
\end{equation}
Since, by induction hypothesis, $G_{n-1}(F_n X)_{r_1}\in\ftawg{n-1}$ we also have a commuting diagram
\begin{equation}\label{eq-pro-fta-3}
\xymatrix@C=50pt@R=40pt{
\{G_{n-1}(F_n X)_{r_1}\}_{\uk' 0} \ar[r] & \{G_{n-1}(F_n X)_{r_1}\}_{\ur' 0}\\
\{G_{n-1}(F_n X)_{r_1}\}^d_{\uk' 0} \ar[r]\ar[u] & \{G_{n-1}(F_n X)_{r_1}\}^d_{\ur' 0}\ar[u]
}
\end{equation}
Combining \eqref{eq-pro-fta-1}, \eqref{eq-pro-fta-2}, \eqref{eq-pro-fta-3} we obtain a commuting diagram
\begin{equation*}
\xymatrix@C=40pt@R=35pt{
(G_n X)_{\uk 0} \ar[r] & (G_n X)_{\ur 0}\\
(G_n X)^d_{\uk 0} \ar[r]\ar[u] & (G_n X)^d_{\ur 0}\ar[u]
}
\end{equation*}
This shows that $G_n X\in\ftawg{n}$. Let $F:X\rw Y$ be a morphism in $\catwg{n}$. Then
 \begin{equation*}
  (F_n F)_{k_1}:(F_n X)_{k_1}\rw (F_n Y)_{k_1}
 \end{equation*}
 is a morphism in $\catwg{n-1}$. Thus by induction hypothesis it induces a morphism
\begin{equation*}
    G_{n-1}(F_n X)_{k_1}\rw G_{n-1}(F_n Y)_{k_1}
\end{equation*}
such that the following diagram commutes:
\begin{equation*}
\xymatrix@R=35pt{
(G_n X)_{\uk 0}=\{G_{n-1}(F_n X)_{k_1}\}_{\uk' 0} \ar[r] & \{G_{n-1}(F_n Y)_{k_1}\}_{\uk' 0}=(G_n Y)_{\uk 0}\\
(G_n X)^d_{\uk 0}\ar[r]\ar[u] & (G_n Y)^d_{\uk 0}\ar[u]
}
\end{equation*}
This shows that $G_n F$ is a morphism in $\ftawg{n}$. In conclusion
\begin{equation*}
    G_n:\catwg{n}\rw \ftawg{n}\;.
\end{equation*}
The fact that $(G_n X)_k=G_{n-1}(F_n X)_k\in\catwg{n-1}$ follows by induction.
\bk

b) The morphism $G_n X\rw X$ is given levelwise  by $G_{n-1} X_k\rw X_k$; this is an $(n-1)$-equivalence for each $k$, hence $G_n X\rw X$ is a $n$-equivalence by Lemma 4.8 of \cite{Pa3}.
\bk

c) Let $F:X\rw Y$ be an $n$-equivalence in $\catwg{n}$. There is a commuting diagram
\begin{equation*}
\xymatrix@R=30pt@C=40pt{
G_n X \ar^{G_n F}[r] \ar[d] & G_n Y \ar[d]\\
X \ar_{F}[r] & Y
}
\end{equation*}
in which the vertical maps and the bottom horizontal map are $n$-equivalences. By Proposition 4.11 of \cite{Pa3} it follows that $G_n F$ is also an $n$-equivalence.
\bk

d) This follows immediately by the analogous properties of $F_n$ and by the inductive hypothesis.

\end{proof}
%
\section{From $\ftawg{n}$ to Tamsamani $n$-categories}\label{sec-fta-to-tam}
In this section we define a discretization functor
\begin{equation*}
    \D{n}:\ftawg{n}\rw \ta{n}
\end{equation*}
and we study its properties. The idea of the functor $\D{n}$ is to replace the homotopically discrete sub-objects in $X\in\ftawg{n}$ by their discretization, thus recovering the globularity condition. Because of the canonical property of the homotopically discrete objects in $\ftawg{n}$, this can be done in a functorial way.

The construction of $\D{n}$ is inductive, and we first need to discretize the structure at level 0, as follows.
\begin{definition}\label{def-fta-tam-1}
    Let $\rz:\ftawg{n}\rw\funcat{}{\ftawg{n-1}}$ be given by
    \begin{equation*}
        (\rz X)_k=
        \left\{
          \begin{array}{ll}
            X_0^d, & k=0 \\
            X_k, & k>0\;.
          \end{array}
        \right.
    \end{equation*}
    The face operators $\pt'_0,\pt'_1:X_1\rightrightarrows X_0^d$ are given by $\pt'_i=\zg\pt_i$, $i=0,1$ and the degeneracy $\zs':X_0^d\rw X_1$ by $\zs'=\zs\zg'$ where $\pt_0,\pt_1,\zs$ are the corresponding face and degeneracy operators of $X$, $\zg:X_0\rw X_0^d$ is the discretization map and $\zg':X^d_0\rw X_0$ is a section.
\end{definition}
\begin{remark}\label{rem-fta-tam-1}
    By definition of $\ftawg{n}$, given $f:X\rw Y$ in $\ftawg{n}$ there is a commuting diagram
    \begin{equation}\label{eq-rem-fta-tam-1}
    \xymatrix@C=40pt{
    X_0 \ar^{f_0}[r] & Y_0 \\
    X_0^d \ar^{\zg'(X_0)}[u] \ar[r] & Y_0^d \ar_{\zg'(Y_0)}[u]
    }
    \end{equation}
    and this induces a morphism in $\funcat{}{\ftawg{n-1}}$ $\rz f:\rz X\rw \rz Y$. Thus $\rz$ is a functor. Note that while $\rz X$ could be defined for any $X\in\tawg{n}$, given a morphism $f$ in $\tawg{n}$ since in general \eqref{eq-rem-fta-tam-1} does not commute, one cannot define $\rz f$ as above.
\end{remark}
\begin{lemma}\label{lem-fta-tam-1}
Let $\rz$ be as in Definition \ref{def-fta-tam-1} then
\begin{itemize}
  \item [a)] $\rz:\ftawg{n}\rw \ftawg{n}$.\mk

  \item [b)] $\rz$ is identity on discrete objects and commutes with pullbacks over discrete objects.\mk

  \item [c)] $\p{n}\rz X=\rz \p{n} X$, \quad $\q{n}\rz X=\rz \q{n} X$.\mk

  \item [d)] $\rz$ preserves $n$-equivalences.
\end{itemize}
\end{lemma}
\begin{proof}
By induction on $n$. Let $X\in \ftawg{2}$. Then $(\rz X)_0=X_0^d$ is discrete while for each $k\geq 2$ the Segal maps
\begin{equation*}
    (\rz X)_k=X_k \xrw{\sim} \pro{X_1}{X_0^d}{k}=\pro{(\rz X)_1}{(\rz X)_0}{k}
\end{equation*}
are equivalences of categories since $X\in \ftawg{2}$. Thus, by definition, $\rz X\in\ta{2}\subset \ftawg{2}$.

The proof of b) is immediate. We have
\begin{equation*}
    \p{2}\rz X =\p{2}X = \rz \p{2} X
\end{equation*}
and similarly for $\q{2}\rz X$, so c) holds. Given a 2-equivalence $f:X\rw Y$ in $\ftawg{2}$, for each $a,b\in X_0^d$, $(\rz f)(a,b)=f(a,b)$ is a 2-equivalence and $\p{2}\rz f=\p{2} f$ is a 2-equivalence. So by definition $\p{2}f$ is a 2-equivalence.

Suppose, inductively, that the lemma holds for $(n-1)$ and let $X\in \ftawg{n}$.

\bk

a) Note that $\rz X\in\tawg{n}$. In fact, by construction $(\rz X)_k\in\tawg{n-1}$ for all $k\geq 0$. For each $k\geq 2$ the induced Segal maps for $\rz X$ are:
\begin{equation*}
    (\rz X)_k=X_k \rw \pro{X_1}{X_0^d}{k}=\pro{(\rz X)_1}{(\rz X)_0}{k}
\end{equation*}
and these are $(n-1)$-equivalences because $X\in \ftawg{n}$. To show that $\rz X \in \ftawg{n}$ note that, given $\uk=(\seqc{k}{1}{s})$ and $\ur=(\seqc{r}{1}{s})$ in $\dop{s}$, $1\leq s \leq n-2$ and a morphism $\uk\rw\ur$ in $\dop{s}$, the following diagram commutes
\begin{equation}\label{eq-lem-fta-tam-1}
\xymatrix@C=40pt{
(\rz X)_{\uk 0} \ar[r] & (\rz X)_{\ur 0} \\
(\rz X)^d_{\uk 0} \ar[r]\ar[u] & (\rz X)^d_{\ur 0} \ar[u]
}
\end{equation}
In fact, if $k_1\neq 0$ and $r_1\neq 0$ diagram \eqref{eq-lem-fta-tam-1} coincides with
\begin{equation*}
\xymatrix@C=40pt{
X_{\uk 0} \ar[r] & X_{\ur 0} \\
X^d_{\uk 0} \ar[r]\ar[u] & X^d_{\ur 0} \ar[u]
}
\end{equation*}
and this commutes since $X\in \ftawg{n}$.

If $k_1=0$ and $r_1\neq 0$ diagram \eqref{eq-lem-fta-tam-1} coincides with
\begin{equation}\label{eq-lem-fta-tam-2}
\xymatrix@C=40pt{
X_0^d \ar[r] \ar@{=}[d] & X_{\ur 0}\\
X_0^d \ar[r]  & X^d_{\ur 0} \ar[u]
}
\end{equation}
which is the composite of
\begin{equation}\label{eq-lem-fta-tam-3}
\xymatrix@R=35pt{
& X_{0 k_2...k_s 0} \ar[r] & X_{\ur 0}\\
X_0^d \ar[ur]\ar[r] & X^d_{0 k_2...k_s 0} \ar[r]\ar[u] & X^d_{\ur 0} \ar[u]
}
\end{equation}
The right square in \eqref{eq-lem-fta-tam-3} commutes since $X\in \ftawg{n}$ and the left triangle commutes by construction. Thus \eqref{eq-lem-fta-tam-2} commutes.

If $k_1\neq 0$ and $r_1=0$ diagram \eqref{eq-lem-fta-tam-1} is
\begin{equation*}
\xymatrix@C=40pt{
X_{\uk 0} \ar[r]  & X^d_{\ur 0}\\
X^d_{\uk 0} \ar[r] \ar[u] & X^d_{\ur 0} \ar@{=}[u]
}
\end{equation*}
and this commutes since it is the composite of
\begin{equation*}
\xymatrix@R=35pt@C=40pt{
X_{\uk 0} \ar[r]  & X_{\ur 0} \ar[dr] & \\
X^d_{\uk 0} \ar[r] \ar[u] & X^d_{\ur 0} \ar[u] \ar@{=}[r] &  X^d_{\ur 0}
}
\end{equation*}
where the left square commutes since $X\in \ftawg{n}$ and the right triangle commutes since $X^d_{\ur 0}\rw X_{\ur 0}$ is a section of $X_{\ur 0}\rw X^d_{\ur 0}$. Hence we conclude that $\rz X\in \ftawg{n}$.

Given $F:X\rw Y$ in $\ftawg{n}$ we have the commuting diagram
\begin{equation}\label{eq-lem-fta-tam-4}
\xymatrix@C=40pt{
(\rz X)_{\uk 0} \ar[r] & (\rz Y)_{\uk 0} \\
(\rz X)^d_{\uk 0} \ar[r]\ar[u] & (\rz Y)^d_{\ur 0} \ar[u]
}
\end{equation}
In fact, when $k_1\neq 0$ this coincides with
\begin{equation*}
\xymatrix@R=35pt@C=40pt{
X_{\uk 0} \ar[r] & Y_{\uk 0} \\
X^d_{\uk 0} \ar[r]\ar[u] & Y^d_{\ur 0} \ar[u]
}
\end{equation*}
which commutes since $F$ is a morphism in $\ftawg{n}$. When $k_1=0$ diagram \eqref{eq-lem-fta-tam-4} is
\begin{equation*}
\xymatrix@R=35pt@C=40pt{
X_0^d \ar[r] & Y_0^d \\
X^d_{0} \ar@{=}[u]\ar[r] & Y^d_{0}\ar@{=}[u]
}
\end{equation*}
In conclusion, diagram \eqref{eq-lem-fta-tam-4} commutes, showing that $\rz F$ is a morphism in $\ftawg{n}$.
\bk

b) This is immediate by the definition of $\rz$ since, if $X\rw Z \lw Y$ is a pullback in $\ftawg{n}$ with $Z$ discrete, $(X\tiund{Z}Y)^d_0=X_0^d\tiund{z}Y_0^d$ by Lemma 3.10 of \cite{Pa1}.

\bk

c) For each $k>0$,
\begin{equation*}
    (\p{n}\rz X)_k=\p{n-1}(\rz X)_k=\p{n-1}X_k =\rz (\p{n}X)_k\;.
\end{equation*}
When $k=0$,
\begin{equation*}
    (\p{n}\rz X)_0= X_0^d =(\p{n}X)^d_0=(\rz\p{n}X)_0\;.
\end{equation*}
Similarly for $\q{n}X$.

\bk

d) Let $f:X\rw Y$ be an $n$-equivalence. For all $a,b\in (\rz X)^d_0=X_0^d$, $(\rz f)(a,b)=f(a,b)$ is a $(n-1)$-equivalence. Also by c) and the induction hypothesis $\p{n}\rz f=\rz \p{n}f$ is a $(n-1)$-equivalence. Thus by definition $\rz f$ is an $n$-equivalence.
\end{proof}
\begin{proposition}\label{pro-fta-tam-1}
There is a functor
\begin{equation*}
    \Dn:\ftawg{n}\rw \ta{n}
\end{equation*}
defined inductively by
\begin{equation*}
    D_2=\rz, \qquad \Dn=\ovl{D}_{n-1}\circ \rz \qquad \text{for }\; n>2
\end{equation*}
where $\rz$ is as in Lemma \ref{lem-fta-tam-1}, such that
\begin{itemize}
  \item [a)] $\Dn$ is identity on objects and commutes with pullbacks over discrete objects.\mk

  \item [b)] $\p{n}\Dn X=\Dnm \p{n} X$,\quad $\q{n}\Dn X=\Dnm \q{n} X$.\mk

  \item [c)] $\Dn$ preserves $n$-equivalences.
\end{itemize}
\end{proposition}
\begin{proof}
By induction on $n$. It holds for $n=2$ by Lemma  \ref{lem-fta-tam-1}. Suppose, inductively, that it holds for $(n-1)$ and let $X\in\ftawg{n}$. Then by induction hypothesis a)
\begin{equation*}
    (\Dn X)_k =
    \left\{
      \begin{array}{ll}
        \Dnm X_0^d=X_0^d & k=0 \\
        \Dnm X_k, & k>0\;.
      \end{array}
    \right.
\end{equation*}
Thus by induction hypothesis $(\Dn X)_k\in\ta{n-1}$ for all $k\geq 0$ with $(\Dn X)_0$ discrete.

To show that $\Dn X\in\ta{n}$ it remains to show that the Segal maps are $(n-1)$-equivalences. Since $X\in\ftawg{n}$, for each $k\geq 2$ the map
\begin{equation*}
    \mu_k: X_k\rw \pro{X_1}{X_0^d}{k}
\end{equation*}
is a $(n-1)$-equivalence. By inductive hypotheses a) and c) this induces a $(n-1)$-equivalence
\begin{align*}
&\Dnm \mu_k:\Dnm X_k=(\Dn X)_k \rw \Dnm(\pro{X_1}{X_0^d}{k})\cong\\
& \cong \pro{(\Dn X)_1}{(\Dn X)_0}{k}\;.
\end{align*}
This shows that the Segal maps of $\Dn X$ are $(n-1)$-equivalences. We conclude that $\Dn X\in\ta{n}$.
\bk

a) This follows from Lemma \ref{lem-fta-tam-1} and the inductive hypothesis.
\bk

b) For each $k\geq 0$, by inductive hypothesis,
\begin{align*}
&(\p{n}\Dn X)_k = \p{n-1}(\Dn X)_k = \p{n-1}\Dnm X_k= \\
& = D_{n-2}\p{n-1}X_k = (\Dnm \p{n}X)_k\;.
\end{align*}
The proof for $\q{n}\Dn X$ is similar.
\bk

c) Let $f:X\rw Y$ be an $n$-equivalence in $\ftawg{n}$. For each $a,b\in(\Dn X)_0=X_0^d$ we have
\begin{equation*}
    (\Dn f)(a,b)=\Dnm f(a,b)
\end{equation*}
and this is a $(n-1)$-equivalence by the inductive hypothesis applied to the $(n-1)$-equivalence $f(a,b)$. Further, by b) $\p{n}\Dn f=\Dnm \p{n} f$ is also a $(n-1)$-equivalence by inductive hypothesis applied to the $(n-1)$-equivalence $\p{n}f$. This shows that $\Dn f$ is a $n$-equivalence.
\end{proof}
\begin{lemma}\label{lem-fta-tam-2}
Let $X\in\ftawg{n}$ and $\uk\in\dop{n-1}$ be such that $k_j\neq 0$ for all $1\leq j\leq n-1$. Then $(\Dn X)_{\uk}=X_{\uk}$.
\end{lemma}
\begin{proof}
By induction on $n$. It clearly holds for $n=2$ since $D_2 X=\rz X$ has $(D_2 X)_k=X_k$ for all $k\neq 0$. Suppose it holds for $(n-1)$ and let $\uk\in\dop{n-1}$ be as in the hypothesis. Denote $\ur=(\seqc{k}{2}{n-1})$. Then by inductive hypothesis applied to $X_{k_1}$ we have
\begin{equation*}
    (\Dn X)_{\uk}=(\Dnm X_{k_1})_{\ur}=(X_{k_1})_{\ur}=X_{\uk}\;.
\end{equation*}
\end{proof}
\begin{proposition}\label{pro-fta-tam-2}
Let $X\in\ftawg{n}$, then $\Qn\Dn X\cong\Qn X$.
\end{proposition}
\begin{proof}
By induction on $n$. Let $X\in\ftawg{2}$. It is immediate that $\rz X$ and $X$ satisfy the hypotheses of Lemma \ref{lem-property-lta} so that
\begin{equation*}
    Tr_2 \rz X\cong Tr_2 X\;.
\end{equation*}
Hence
\begin{equation*}
    Q_2D_2 X=\St Tr_2 \rz X\cong \St Tr_2 X=Q_2 X\;.
\end{equation*}
Suppose, inductively, the statement holds for $(n-1)$ and let $X\in\ftawg{n}$. We claim that $P_n\Dn X$ and $P_n X$ satisfy the hypotheses of Lemma \ref{lem-property-lta}. In fact, by definition of $P_n$ there is a pullback in $\funcat{n-1}{\Cat}$
\begin{equation}\label{eq-lem-fta-tam-5}
\xymatrix@R=35pt{
P_n\Dn X \ar[r]\ar[d] & \Dn X\ar[d]\\
\di{n}\Qnm \q{n}\Dn X \ar[r] & \di{n}\q{n}\Dn X
}
\end{equation}
On the other hand, by Proposition \ref{pro-fta-tam-1} and the inductive hypothesis
\begin{equation*}
    \Qnm \q{n}\Dn X=\Qnm \Dnm \q{n}X=\Qnm \q{n}X
\end{equation*}
so that \eqref{eq-lem-fta-tam-5} coincides with
\begin{equation}\label{eq-lem-fta-tam-6}
\xymatrix@R=35pt{
P_n\Dn X \ar[r]\ar[d] & \Dn X\ar[d]\\
\di{n}\Qnm \q{n} X \ar[r] & \di{n}\q{n}\Dn X
}
\end{equation}
We also have a pullback in $\funcat{n-1}{\Cat}$
\begin{equation}\label{eq-lem-fta-tam-7}
\xymatrix@R=35pt{
P_n X \ar[r]\ar[d] & X\ar[d]\\
\di{n}\Qnm \q{n} X \ar[r] & \di{n}\q{n} X
}
\end{equation}
Since pullbacks in $\funcat{n-1}{\Cat}$ are computed pointwise, for each $\uk\in\dop{n-1}$ \eqref{eq-lem-fta-tam-6} and \eqref{eq-lem-fta-tam-7} give rise to pullbacks in $\Cat$
\begin{equation}\label{eq-lem-fta-tam-8}
\xymatrix@R=35pt{
(P_n \Dn X)_{\uk} \ar[r]\ar[d] & (\Dn X)_{\uk}\ar[d]\\
d(\Qnm \q{n} X)_{\uk} \ar[r] & dq(\Dn X)_{\uk}
}
\end{equation}
\begin{equation}\label{eq-lem-fta-tam-9}
\xymatrix@R=35pt{
(P_n X)_{\uk} \ar[r]\ar[d] & (X)_{\uk}\ar[d]\\
d(\Qnm \q{n} X)_{\uk} \ar[r] & dq(X)_{\uk}
}
\end{equation}
If $k_j=0$ for some $j$, then $(\Dn X)_{\uk}$ is discrete (since $\Dn X\in\ta{n}$) hence the right vertical map in \eqref{eq-lem-fta-tam-8} is an isomorphism, and thus so is the left vertical map in \eqref{eq-lem-fta-tam-8}. That is
\begin{equation*}
    (P_n \Dn X)_{\uk}\cong d(\Qnm \q{n} X)_{\uk}
\end{equation*}
so that
\begin{equation}\label{eq-lem-fta-tam-10}
    p(P_n \Dn X)_{\uk}\cong (\Qnm \q{n} X)_{\uk}\;.
\end{equation}
Further, $X_{\uk}\in\cathd{}$ so $q X_{\uk}=p X_{\uk}$. Thus from \eqref{eq-lem-fta-tam-9}, using the fact that $p$ commutes with pullbacks over discrete objects, we obtain
\begin{equation}\label{eq-lem-fta-tam-11}
    p(P_n X)_{\uk}\cong (\Qnm \q{n} X)_{\uk}\;.
\end{equation}
It follows from \eqref{eq-lem-fta-tam-10} and \eqref{eq-lem-fta-tam-11} that
\begin{equation*}
    (P_n \Dn X)^d=dp(P_n \Dn X)_{\uk}\cong dp(P_n X)_{\uk}=(P_n X)^d_{\uk}\;.
\end{equation*}
This proves hypothesis i) in Lemma \ref{lem-property-lta}. Let $\uk\in\dop{n-1}$ be such that $k_j\neq 0$ for all $j$. Then by Lemma \ref{lem-fta-tam-2} $(\Dn X)_{\uk}=X_{\uk}$ for all $\uk$. Hence the right vertical maps in \eqref{eq-lem-fta-tam-8} and \eqref{eq-lem-fta-tam-9} coincide. It follows that
\begin{equation*}
    (P_n\Dn X)_{\uk} \cong (P_n X)_{\uk}\;.
\end{equation*}
This proves hypothesis ii) of Lemma \ref{lem-property-lta} that
\begin{equation*}
    Tr_n P_n \Dn X\cong Tr_n P_n X
\end{equation*}
which implies
\begin{equation*}
    \Qn \Dn X \cong \St Tr_n P_n \Dn X\cong \St Tr_n P_n X\cong \Qn X\;.
\end{equation*}
\end{proof}
\begin{definition}\label{def-disc-func}
Define the discretization functor $\Discn:\catwg{n}\rw \ta{n}$ to be the composite
\begin{equation*}
    \catwg{n}\xrw{G_n}\ftawg{n}\xrw{\Dn}\ta{n}\;.
\end{equation*}
\end{definition}
\begin{theorem}\label{the-disc-func}
\
  \begin{itemize}
    \item [a)] $\Discn$ is identity on discrete objects and commutes with pullbacks over discrete objects.\mk

    \item [b)] For each $X\in\catwg{n}$ there is a zig-zag of $n$-equivalences in $\tawg{n}$ between $X$ and $\Discn X$.\mk

    \item [c)] $\Discn$ preserves $n$-equivalences.
  \end{itemize}
\end{theorem}
\begin{proof}
\

a) This follows from the fact that the same is true for $G_n$ and $\Dn$ (see Proposition \ref{pro-fta-1} and Proposition \ref{pro-fta-tam-1}).\bk

b) Let $X\in\catwg{n}$, then by Proposition \ref{pro-fta-tam-2}
\begin{equation*}
    \Qn\Discn X=\Qn \Dn G_n X = \Qn G_n X\;.
\end{equation*}
Hence by Theorem \ref{the-funct-Qn} there are $n$-equivalences in $\tawg{n}$
\begin{equation*}
    \Discn X \lw \Qn \Discn X = \Qn G_n X\rw G_n X\;.
\end{equation*}
On the other hand by Proposition \ref{pro-fta-1} there is an $n$-equivalence in $\tawg{n}$, $G_n X\rw X$. So by composition we obtain $n$-equivalences
\begin{equation*}
    \Discn X \lw \Qn \Discn X \rw X
\end{equation*}
as required.
\bk

c) This follows from the fact that the same is true for $G_n$ and $\Dn$.
\end{proof}
\begin{corollary}\label{cor-the-disc-func}
    The functors
    \begin{equation*}
        \Qn:\ta{n}\leftrightarrows \catwg{n}:\Discn
    \end{equation*}
    induce an equivalence of categories after localization with respect to the $n$-equivalences
    \begin{equation*}
        \ta{n}\bsim^n\;\simeq \; \catwg{n}\bsim^n
    \end{equation*}
\end{corollary}
\begin{proof}
Let $X\in\catwg{n}$. By by Theorem \ref{the-funct-Qn} and Proposition \ref{pro-fta-1} there are $n$-equivalences
\begin{equation*}
    \Qn\Discn X=\Qn \Dn G_n X=\Qn G_n X\rw G_n X\rw X\;.
\end{equation*}
So there is an $n$-equivalence in $\catwg{n}$
\begin{equation*}
    \Qn \Discn X\rw X\;.
\end{equation*}
It follows that $\Qn \Discn X\cong X$ in $\catwg{n}\bsim^n$. Let $Y\in\ta{n}$. By Theorem \ref{the-funct-Qn} there are $n$-equivalences in $\tawg{n}$
\begin{equation*}
    \Discn \Qn Y\lw \Qn \Discn \Qn Y \rw \Qn Y \rw Y\;.
\end{equation*}
Composing this with the $n$-equivalences
\begin{equation*}
    Z=G_n \Qn \Discn \Qn Y\rw \Qn \Discn \Qn Y
\end{equation*}
we obtain $n$-equivalences in $\tawg{n}$
\begin{equation*}
\Discn \Qn Y\lw Z \rw Y\;.
\end{equation*}
Since $Z\in\ftawg{n}$ and $\Discn \Qn Y\in\ta{n}$, $Y\in\ta{n}$ and $\ftawg{n}\subset \ta{n}$, this is a zig-zag of $n$-equivalences in $\ftawg{n}$ . Therefore we can apply $\Dn$ to the above zig-zag and obtain a zig-zag of $n$-equivalences in $\ta{n}$
\begin{equation*}
    \Discn \Qn Y = \Dn \Discn \Qn Y \lw \Dn Z \rw \Dn Y = Y\;.
\end{equation*}
It follows that $ \Discn \Qn Y \cong Y$ in $\ta{n}\bsim^n$.

\end{proof}


\section{Groupoidal weakly globular n-fold categories}\label{sec-group-wg-nfol-cat}

In this section we introduce the category $\gtawg{n}$ of groupoidal weakly globular $n$-fold categories and we show that it is an algebraic model of $n$-types. This means that weakly globular $n$-fold categories satisfy the homotopy hypothesis.
\begin{definition}\label{def-gta-1}
    The full subcategory $\gtawg{n}\subset\tawg{n}$ of groupoidal weakly globular \nfol categories is defined inductively as follows.

    For $n=1$, $\gtawg{1}=\Gpd$. Note that $\cathd{}\subset \gtawg{1}$. Suppose inductively we defined $\gtawg{n-1}\subset\tawg{n-1}$ such that
    \begin{itemize}
      \item [i)] $X_k\in \gtawg{n-1}$ for all $k\geq 0$.\mk

      \item [ii)] $\p{n}X\in \gtawg{n-1}$.
    \end{itemize}
\end{definition}
\begin{lemma}\label{lem-gta-1}
    Let $f:X\rw Y$ be an equivalence in $\tawg{n}$
    \begin{itemize}
      \item [i)] If $Y\in\gtawg{n}$ then $X\in\gtawg{n}$.\mk

      \item [ii)] If $X\in\gtawg{n}$ then $Y\in\gtawg{n}$.
    \end{itemize}
\end{lemma}
\begin{proof}
By induction on $n$. The case $n=1$ holds since a category equivalent to a groupoid is itself a groupoid. Suppose, inductively, that the lemma holds for $n-1$ and let $f:X\rw Y$ be an $n$-equivalence.

i) For each $a,b\in X_0^d$ the map
\begin{equation*}
    f(a,b):X(a,b)\rw Y(fa,fb)
\end{equation*}
is a $(n-1)$-equivalence in $\tawg{n-1}$ with $Y(fa,fb)\in\gtawg{n-1}$. So by induction hypothesis $X(a,b)\in \gtawg{n-1}$. Since
\begin{equation*}
    X_1=\uset{a,b\in X_0^d}{\coprod}X(a,b)
\end{equation*}
it follows that $X_1\in\gtawg{n-1}$. We also have
\begin{equation*}
    \tens{X_1}{X_0^d}=\uset{a,b,c\in X_0^d}{\coprod}X(a,b)\times X(b,c)
\end{equation*}
so that $\tens{X_1}{X_0^d}\in\gtawg{n-1}$. Similarly one can show that $\pro{X_1}{X_0^d}{k}\in\gtawg{n-1}$ for all $k\geq 2$. In conclusion, $X_k\in\gtawg{n-1}$ for all $k\geq 0$.

By definition there is a $(n-1)$-equivalence
\begin{equation*}
    \p{n}f:\p{n}X\rw\p{n}Y
\end{equation*}
with $\p{n}Y\in\gtawg{n-1}$ since by hypothesis $Y\in\gtawg{n}$. Hence by inductive hypothesis $\p{n}X\in\gtawg{n-1}$. We conclude that $X\in\gtawg{n-1}$.
\bk

ii) The proof is completely similar to one of i).

\end{proof}
\begin{remark}\label{rem-gta-1}
    It follows immediately from the definition of $\gtawg{n}$ that the embedding
    \begin{equation*}
        J_n:\tawg{n}\hookrightarrow \funcat{n-1}{\Cat}
    \end{equation*}
    restricts to the embedding
    \begin{equation*}
        J_n:\gtawg{n}\hookrightarrow\funcat{n-1}{\Gpd}\;.
    \end{equation*}
    Since $p=q:\Gpd\rw\Set$ it follows that for each $X\in\gtawg{n}$ there is a morphism, natural in $X$,
    \begin{equation*}
        X\rw\di{n}\p{n}X\;.
    \end{equation*}
\end{remark}
\begin{definition}\label{def-gta-2}
    The category $\gcatwg{n}\subset\catwg{n}$ of groupoidal weakly globular \nfol categories is the full subcategory of $\catwg{n}$ whose objects $X$ are in $\gtawg{n}$.

    The category $\gta{n}\subset \ta{n}$ of groupoidal Tamsamani $n$-categories is the full subcategory of $\ta{n}$ whose objects $X$ are in $\gtawg{n}$.
\end{definition}
\begin{remark}\label{rem-gta-2}
    The following facts are immediate from the definitions:
    \begin{itemize}
      \item [a)] $X\in\gcatwg{n}$ (resp. $X\in\gta{n}$) if and only if for each $k\geq 0$ $X_k\in\gcatwg{n-1}$ (resp. $X_k\in\gta{n-1}$) and $\p{n}X\in\gcatwg{n}$ (resp. $\p{n}X\in\gta{n-1}$).\mk

      \item [b)] Let $f:X\rw Y$ be an $n$-equivalence in $\tawg{n}$ and suppose that $Y\in\gtawg{n}$. Then if $X\in\catwg{n}$ it is $X\in\gcatwg{n}$  and  if $X\in\tawg{n}$ then $X\in\gtawg{n}$.
    Similarly if $f$ is an $n$-equivalence in $\tawg{n}$ and $X\in\gtawg{n}$.

    \end{itemize}
\end{remark}
\begin{corollary}\label{cor-gta-1}
The following facts hold:
    \begin{itemize}
      \item [a)] The functor
      \begin{equation*}
        \Qn:\tawg{n}\rw\catwg{n}
      \end{equation*}
      restricts to a functor
      \begin{equation*}
        \Qn:\gtawg{n}\rw\gcatwg{n}
      \end{equation*}
      such that for each $X\in\gtawg{n}$ there is a $n$-equivalence in $\gtawg{n}$ $s_n(X):\Qn X\rw X$.\mk

      \item [b)] The functor
      \begin{equation*}
        \Discn:\catwg{n}\rw\ta{n}
      \end{equation*}
      restricts to a functor
      \begin{equation*}
        \Discn:\gcatwg{n}\rw\gta{n}
      \end{equation*}
      such that for each $X\in\gcatwg{n}$ there is a zig-zag of $n$-equivalences in $\gtawg{n}$ between $X$ and $\Discn X$.
    \end{itemize}
\end{corollary}
\newpage
\begin{proposition}\label{pro-gta-1}
    The functors
    \begin{equation*}
        \Qn:\gta{n}\leftrightarrows \gcatwg{n}:\Discn
    \end{equation*}
    induce an equivalence of categories after localization with respect to the $n$-equivalences
    \begin{equation*}
        \gta{n}\bsim^n \;\simeq\; \gcatwg{n}\bsim^n\;.
    \end{equation*}
\end{proposition}
\begin{proof}
Let $X\in\gcatwg{n}$. As in the proof of Corollary \ref{cor-the-disc-func} there is an $n$-equivalence in $\catwg{n}$
\begin{equation*}
    \Qn\Discn X \rw X\;.
\end{equation*}
Since $X\in\gcatwg{n}$, by Remark \ref{rem-gta-2}, $\Qn \Discn X\in\gcatwg{n}$, so this is an $n$-equivalence in $\gcatwg{n}$. It follows that
\begin{equation*}
    \Qn\Discn X\cong X
\end{equation*}
in $\gcatwg{n}\bsim^n$.

Let $Y\in\gta{n}$. By the proof of Corollary \ref{cor-the-disc-func} there is a zig-zag of $n$-equivalences in $\ta{n}$
\begin{equation*}
    \Discn \Qn Y\lw  \Dn Z \rw Y\;.
\end{equation*}
Since $Y\in\gta{n}$, from Remark \ref{rem-gta-2} $\Discn\Qn Y\in\gta{n}$ and $\Dn Z\in\gta{n}$ so that this is a zig-zag of $n$-equivalences in $\gta{n}$. It follows that
\begin{equation*}
    \Discn \Qn Y\cong Y
\end{equation*}
in $\gta{n}\bsim^n$.
\end{proof}
\begin{corollary}\label{cor-gta-2}
    There is an equivalence of categories
    \begin{equation*}
        \gcatwg{n}\bsim^n\;\simeq\; \Ho\mbox{\rm\text{(n-types)}}\;.
    \end{equation*}
\end{corollary}
\begin{proof}
By \cite{Ta} there is an equivalence of categories
\begin{equation*}
    \gta{n}\bsim^n\;\simeq\;\Ho\text{(n-types)}\;.
\end{equation*}
Hence by Proposition \ref{pro-gta-1} the result follows.
\end{proof}


\end{document}